\let\SIAMkernelLabel\label
\let\SIAMkernelRefstepcounter\refstepcounter
\AddToHook{package/hyperref/before}{%
  \let\label\SIAMkernelLabel
  \let\refstepcounter\SIAMkernelRefstepcounter
}
\documentclass[final,onefignum,onetabnum]{siamart171218}

\usepackage{lipsum}
\usepackage{amsfonts}
\usepackage{graphicx}
\usepackage{algorithmic}
\usepackage{amsmath}
\usepackage{mathtools}
\usepackage{tkz-graph}
 \DeclareGraphicsExtensions{.pdf}

\let\norm\undefined 
\DeclarePairedDelimiter\norm{\lVert}{\rVert}

\newsiamremark{remark}{Remark}
\newsiamremark{assumption}{Assumption}
\newsiamremark{hypothesis}{Hypothesis}
\crefname{hypothesis}{Hypothesis}{Hypotheses}
\newsiamthm{claim}{Claim}

\headers{Regularization of Primal-Dual Dynamics}{M. Vaquero and J. Cort\'es}

\title{On the Regularization of Primal-Dual Dynamics in time-dependent non-convex Optimization}

\author{Miguel Vaquero\thanks{School of Science and Technology, IE University,
Madrid, Spain (\email{miguel.vaquero@ie.edu}).}
\and Jorge Cort\'es\thanks{Department of Mechanical and Aerospace Engineering,
University of California San Diego, La Jolla, CA 92093, USA
(\email{cortes@ucsd.edu}).}}

\usepackage{amsopn}

\usepackage{booktabs}
\usepackage{tabularx}
\usepackage{placeins}
\usetikzlibrary{positioning,arrows.meta,fit,backgrounds,calc}
\newcolumntype{Y}{>{\raggedright\arraybackslash}X}
\definecolor{guideblue}{RGB}{35,78,121}
\definecolor{guideorange}{RGB}{184,92,38}
\definecolor{guidegreen}{RGB}{48,122,88}
\newcommand{\auditfix}[1]{#1}
\newcommand{\greenfix}[1]{#1}
\title{On the Regularization of Primal--Dual Dynamics in Time-Dependent Nonconvex Optimization}

\ifpdf 
\hypersetup{
  pdftitle={On the Regularization of Primal-Dual Dynamics in Time-Dependent Nonconvex Optimization},
  pdfauthor={Miguel Vaquero and Jorge Cort\'es},
  colorlinks=true,
  linkcolor=black,
  citecolor=black,
  urlcolor=black
}
\fi

\newcommand{\real}{\mathbb{R}}
\newcommand{\realnonneg}{\mathbb{R}_{\ge 0}}
\newcommand{\realpos}{\mathbb{R}_{> 0}}
\newcommand{\until}[1]{\{1,\dots,#1\}}
\newcommand{\Lc}{\mathcal{L}}

\newcommand{\inv}[1]{{#1}^{\operatorname{imp}}}

\newcommand{\oprocendsymbol}{\hbox{$\bullet$}}
\newcommand{\oprocend}{\relax\ifmmode\else\unskip\hfill\fi\oprocendsymbol}

\allowdisplaybreaks

\begin{document}

\maketitle
    
\begin{abstract}
We study continuous-time primal--dual dynamics for tracking solutions of
time-dependent constrained optimization problems. First, for uniformly
strongly convex--strongly concave saddle functions, we establish explicit
finite-time tracking and additive-disturbance bounds. For convex problems,
we then introduce a directional primal regularization, combined with dual
Tikhonov regularization, and quantify the resulting optimization bias and
constraint violation. For smooth nonconvex problems, we couple a proximal
local convexification with a two-time-scale primal--dual flow and obtain a
local tracking result under constraint qualification, second-order
sufficiency, and temporal regularity. Numerical traffic-signal experiments at a single intersection
evaluate the proposed method  under vehicle-throughput and
pedestrian-aware objectives, comparing it with standard signal-control
schemes.
\end{abstract}


\section{Introduction}
\paragraph{Problem and motivation}
 Traffic control, power-system operation, and resource allocation often lead to optimization problems that are both time-varying and nonconvex. As the data evolve, an optimizer may move, but successive problem instances are usually related: a solution found at one time can guide the search for the next. This motivates dynamical optimization methods that track a moving solution.

For the sake of concreteness, consider traffic-signal control at a single intersection. A signal cycle is one complete sequence of the intersection’s signal stages. For each cycle, the controller chooses the duration of each stage, subject to a fixed total cycle time, to maximize the number of vehicles served. Because throughput depends on traffic demand and evolving queues, we do not assume an accurate closed-form relation between stage durations and vehicles served. Instead, we evaluate sampled timing plans with a traffic simulator and fit a differentiable surrogate to guide the optimization; this surrogate need not be convex. Queues left at the end of one cycle carry into the next, creating a sequence of related but changing problems. The timing plan found for one cycle can thus help find a solution for the next.

\paragraph{Primal–dual dynamics are a natural candidate for tracking solutions of such constrained problems} In the convex setting, equilibria of primal–dual dynamics are saddle points of the associated Lagrangian and correspond to primal–dual solutions of the optimization problem. Quantitative tracking estimates hold when the saddle function is  strongly convex in the primal variables and strongly concave in the dual variables. Two difficulties prevent this strongly convex–strongly concave theory from applying directly to the class of problems considered here. First, standard optimization Lagrangians are affine in the dual multipliers and therefore lack the strong concavity required by the tracking analysis; even for a two-variable linear program, the corresponding classical saddle flow can oscillate persistently rather than converge~\cite{TH-IL:14,TH-IL:21}. Second, in a nonconvex constrained problem, a KKT solution generally admits only a local saddle-point interpretation. These two difficulties motivate the regularization and local convexification developed in this paper.

\subsection{Statement of contributions}
The paper addresses four questions, moving from time-varying saddle
flows to convex and nonconvex optimization and then to simulation-based
traffic control.

  \begin{itemize}
  \item \emph{How closely can a primal--dual flow follow a moving saddle
  point?} We consider a time-varying saddle function that is uniformly
  strongly convex in the primal variables and strongly concave in the dual
  variables. When its saddle point moves at bounded speed, we bound the
  tracking error between the primal--dual state and the current saddle point in terms of the initial error, the saddle point's motion,
  and disturbances to the flow.
  
   \item \emph{Can primal--dual dynamics track an optimizer of a
  time-varying convex program while limiting the change caused by
  regularization?} The standard (unregularized) Lagrangian used in optimization does not satisfy the strong
  convexity and strong concavity conditions used in our first result.
  Standard Tikhonov regularization restores these properties but can
  move the optimizer. We instead choose a regularization that exploits the
  structure of the equality constraints. We prove tracking and robustness
  bounds and quantify the resulting changes in the optimizer, objective
  value, and feasibility. In our convex numerical tests, this approach
  gives smaller errors relative to the original optimizer while retaining
  comparable robustness to disturbances.

  \item \emph{Can primal--dual dynamics stay near a moving local
  optimizer of a nonconvex problem?} Such an optimizer need not be a
  global Lagrangian saddle point. We use local convexification near the
  moving optimizer and bound how closely the resulting dynamics follow
  it, provided the solution path satisfies suitable regularity
  conditions and the dynamics start nearby the initial optimizer.

  \item \emph{Can the method use an objective available only through
  simulation?} In a traffic-signal experiment, we fit a differentiable
surrogate to SUMO (\cite{PAL-MB-LBW-JE-YE-YF-RH-LL-JR-PW-EW:18})
  simulations and compare the resulting timing plans with equal-split,
  Webster, and max-pressure controllers under vehicle-throughput and
  pedestrian-aware objectives.
  \end{itemize}

\section{Related work}\label{sec:related-work}

\subsection{Saddle flows and time-varying convex optimization}
Continuous-time primal--dual methods go back to the saddle-point dynamics
of Kose and Arrow, Hurwicz, and Uzawa
\cite{KA-LH-HU:58,TK:56}. For fixed convex programs, Lyapunov and invariance
arguments establish convergence of projected saddle flows under several
combinations of convexity, strict convexity, and constraint regularity
\cite{AC-BG-JC:17-sicon,AC-EM-JC:16-scl,DF-FP:10}. For strongly convex optimization problems with affine constraints, global exponential stability of augmented primal–dual gradient dynamics has been established under suitable assumptions \cite{JC-SKN:19-jnls,GQ-NL:19}.

Two viewpoints in the time-varying convex-optimization literature are
especially relevant here; see \cite{AS-ED-SP-GL-GBG:20} for a broad survey.
Running and prediction--correction methods sample the problem, take a
limited number of iterations at each sample, and bound the resulting
tracking error. Representative constrained methods include first-order
prediction--correction schemes~\cite{AS-ED:17-pc} and feedback-based
projected primal--dual iterations~\cite{AB-ED-AS:18}. A second viewpoint
embeds the optimizer in a continuous-time system. Prediction--correction
interior-point dynamics can obtain vanishing error by using temporal
derivatives of the problem data~\cite{MF-SP-VMP-AR:18}. Without such a
feedforward term, contraction analysis yields an error floor proportional
to the motion of the optimizer. This has been shown for centralized and
distributed equality-constrained primal--dual dynamics
\cite{PCV-SJ-FB:20} and, more generally, for monotone inclusions,
equality-constrained programs, and composite minimization
\cite{AD-VC-AG-GR-FB:25}. Projected primal--dual systems for strongly
convex problems with time-varying inequality constraints are studied in
\cite{RR-ADM-UV:24}, while the dynamic embedding in
\cite{MA-IVK-MH:26} combines an explicit tracking bound with continuous
primal feasibility.

\subsection{Regularization and nonconvex solution tracking}
Primal and dual Tikhonov terms are commonly used to make a
Lagrangian strongly convex in the primal variables and strongly
concave in the multipliers, at the price of shifting its saddle
point~\cite{AB-ED-AS:18,JK-AN-UVS:11}. Koshal,
Nedi\'c, and Shanbhag analyze distributed regularized primal--dual methods
and bound the objective-value error induced by the approximation
\cite{JK-AN-UVS:11}; regularized Lagrangians also underlie online
primal--dual tracking with measurement feedback~\cite{AB-ED-AS:18}.
The convex part of the present article develops the
augmentation--regularization mechanism introduced in our preliminary
conference paper~\cite{MV-JC:18-necsys}. That work considered distributed,
time-varying convex optimization and illustrated the construction on a
dynamic traffic-assignment problem. Relative to it, the present treatment
makes the row-space/null-space geometry explicit and quantifies the
resulting regularization bias.

For fixed nonconvex problems, convexification and proximal primal--dual
methods replace the original problem by a sequence of better behaved local
models. Classical convexification procedures appear in~\cite{DPB:79},
local convergence of sequential convex programming is analyzed in
\cite{QTD-MD:10}, and the proximal primal--dual method of~\cite{MH:16-proxPDA}
converges to stationary solutions for smooth linearly constrained
problems. These are iteration-convergence results for a fixed problem,
rather than tracking results for a moving local solution.

The local behavior of a moving KKT point is instead naturally described by
parametric optimization and generalized equations. Strong regularity gives
a locally single-valued Lipschitz solution map
\cite{ALD-RTR:14,SMR:80,SMR:91}; path following and the singularities at
which a branch can cease to be regular are treated in
\cite{JG-DN-FG:01}. Zavala and Anitescu use this framework to analyze
real-time nonlinear optimization as a generalized equation
\cite{VZ-MA:10}. Our treatment employs the same sensitivity machinery to
construct the local implicit primal--dual map combined with a separate Lyapunov analysis of the coupled flow  to establish  the tracking result.

The closest work in scope is the running regularized primal--dual method of
Tang, Dall'Anese, Bernstein, and Low for time-varying constrained
nonconvex programs~\cite{YT-ED-AB-SL:18}. That work studies both a sampled
running algorithm and its continuous-time differential-inclusion limit,
uses dual Tikhonov regularization, and derives local KKT-trajectory error
bounds and nonbifurcation conditions. The two analyses are complementary.
Tang et al. establish a local contraction condition for a direct
dual-regularized primal--dual update. Here, a moving proximal center first
defines a locally strongly convex envelope; an inner primal--dual flow and
an outer descent then evolve on two time scales. This additional state
leads to different assumptions and to an explicit tube-invariance test and
additive-disturbance estimate.

\subsection{Traffic control and simulation-based optimization}
Traffic-signal control spans several distinct modeling traditions.
Webster's fixed-time formulas~\cite{webster1958} and adaptive systems such
as SCOOT and PRODYN~\cite{JJH-JLF-JT:83,PWH-DIR-RDB-MJR:82} optimize or
update cycle plans using analytical traffic relationships and measured
flows. Network models such as the cell-transmission model
\cite{CFD:94} support mathematical-programming and optimal-control
formulations, whereas max-pressure control uses adjacent queue information
and has a throughput-optimality guarantee under its store-and-forward
model~\cite{VaraiyaMaxpressure2013}. Learning-based controllers, such as
PressLight~\cite{HW-CC-GZ-KW-VG-KX-ZL:19}, use training episodes to learn
reactive signal policies.

The traffic component of this paper is closest to simulation-based
optimization. For instance, Osorio and Bierlaire combine a stochastic traffic simulator,
an analytical queueing model, and a derivative-free trust-region method for
urban signal control~\cite{CO-MB:13}. In contrast, we use SUMO evaluations
\cite{PAL-MB-LBW-JE-YE-YF-RH-LL-JR-PW-EW:18} to construct a radial basis function (RBF) surrogate whose gradient enters the regularized primal--dual flow. The
single-intersection experiments compare this optimization pipeline with
classical traffic-control methods.

\section{Notation and preliminaries}\label{sec:notation}

Table~\ref{tab:notation-guide} collects the symbols that carry the main
algorithmic meaning. Hats denote exact or regularized solution curves;
superscript ``imp'' denotes the local implicit solution of the proximal
subproblem.

\begin{table}[h]
\centering
\caption{Notation used across the three analytical layers.}
\label{tab:notation-guide}
\small
\begin{tabularx}{0.96\textwidth}{@{}p{0.22\textwidth}Y@{}}
\toprule
Symbol & Meaning \\
\midrule
$t$, $\tau$, $\kappa$
& Problem time, algorithm time, and their relation $t=\kappa\tau$ in
  the general saddle-flow result. \\
$x$, $p$, $q$
& Primal variable, equality multiplier, and nonnegative inequality
  multiplier. \\
$\hat x(t),\hat p(t),\hat q(t)$
& Reference KKT or saddle-point curve. Subscripts indicate the
  regularization used to define a nearby curve. \\
$\nu_1,\nu_2,\alpha$
& Equality-direction, null-space, and dual regularization parameters in
  the convex construction of Section~\ref{sec:regularization}. \\
$c,\alpha,\nu$
& Proximal, inner dual, and outer primal regularization parameters in
  the local nonconvex construction of Section~\ref{sec:non_convex_setting}. \\
$\epsilon_1,\epsilon_2$
& Inner primal--dual and outer descent time scales in the nonconvex flow. \\
$z^{\rm imp}(t,\ell)$
& Locally unique primal--dual solution of the regularized proximal
  subproblem at parameter $\ell$. \\
\bottomrule
\end{tabularx}
\end{table}

\subsection{Basic definitions}

\paragraph{Euclidean geometry}
We let $\real$ and $\realnonneg$ denote the set of real and
non-negative real numbers, respectively. In the Euclidean space, we use
$\norm{\cdot}$ and $\norm{\cdot}_p$ to denote the $2$- and {\it $p$-norm},
respectively.  The {\it ball} centered at $x\in\mathbb{R}^n$ with radius $r$ is denoted by $B(x;r)$, that is, $B(x;r) = \{y\in\mathbb{R}^n\textrm{ such that }\norm{ y - x } < r\}$. All vectors are considered as column vectors, and the
scalar product is denoted by $\langle \cdot,\cdot\rangle$. The
\emph{kernel} of a matrix $A$ is the subspace of vectors $v$ satisfying
$Av=0$. We denote the minimum and maximum eigenvalues of $A$ by
$\lambda_{\min}(A)$ and $\lambda_{\max}(A)$, respectively. Given
$v\in\mathbb{R}^n$, its components are $v^i$, $i=1,\ldots,n$.

\paragraph{Convexity}
A set $X
\subset \real^m $ is {\it convex} if, for any $x,y \in X$, the segment
$[x,y]=\{\lambda x + (1-\lambda) y\; |\; \lambda \in [0,1]\}$ belongs
to $X$. The {\it normal cone} of a convex set $X$ is $ N_X(z) = \{w \in
\real^m \; | \; \langle w,\overline{z} - z\rangle \leq 0\textrm{ for
  all }\overline{z} \in X \} $ if $z\in X$ and $\emptyset$ otherwise. A function $f:X \rightarrow \real$ on a convex set $X \subset \real^m$
is convex if $f(\lambda x + (1-\lambda) y) \leq \lambda f(x) + (1 -
\lambda)f(y)$ for all $x,y \in X$ and $\lambda \in [0,1]$. It is
{\it strictly convex} if the inequality is strict whenever $x\ne y$ and
$\lambda\in(0,1)$. A function $f$ is {\it $\nu$-strongly convex} when
$x\mapsto f(x)-\nu\norm{x}^2/2$ is convex, $f$ is concave when $-f$ is
convex, and $f$ is $\nu$-strongly concave when $-f$ is $\nu$-strongly
convex.

\paragraph{Smooth and nonsmooth calculus}
For a
vector-valued function $f:\real^m \rightarrow \real^{n}$, we let
$\nabla f$ be the $m \times n$-matrix whose columns are the gradients
$\nabla f^i$ of the component functions $f^i$, $i \in \until{n}$.  A
function $f:\real^m\rightarrow \real$ is {\it locally Lipschitz} at $x \in
\real^m$ if there exist $\delta_x, L_x > 0$ such that $|f(y) -
f(z)|\leq L_x\norm{y - z}$ for all $y, z $ within a distance
$\delta_x$ of~$x$. The function $f$ is locally Lipschitz if it is locally
Lipschitz at $x$, for all $x \in \real^m$, and globally
Lipschitz if there is a constant $L$ such that $|f(y) - f(z)|\leq
L\norm{y - z}$ for all $y, z \in \real^m$. By Rademacher's
Theorem~\cite{FHC:83}, locally Lipschitz functions are differentiable
almost everywhere. At the set of points where $f$ is not
differentiable, denoted $\Omega_f$, the generalized gradient is
defined as $ \partial f(x) = co\{\lim\limits_{i\rightarrow \infty}
\nabla f(x_i):x_i\rightarrow x,x_i\not\in S\cup \Omega_f\} $, where
$co$ denotes the {\it convex hull} and $S\subset \real^m$ is any set of
Lebesgue measure zero.  Given a function of several arguments
$f:\real^{m_1}\times \real^{m_2}\rightarrow \real$, we refer to its
partial derivative in the first argument by $\nabla_x f(x,y)$ or
$\displaystyle\frac{\partial f}{\partial x}(x,y)$ indistinctly, and so
on. We omit reference to the evaluation point when there is no
confusion.

\paragraph{Saddle points}
The point $(\overline{x},\overline{y})$ is a saddle point
of $f:X \times Y\rightarrow \mathbb{R}$ if
$f(\overline{x},y)\leq f(\overline{x},\overline{y})\leq
f(x,\overline{y})$ for all $x \in X$ and $y \in
Y$. A function of two arguments $f:X\times Y\rightarrow \mathbb{R}$ is {\it convex-concave} when the function $f(\cdot{ },y)$ is convex for all $y\in Y$ and $f(x,\cdot{ })$ is concave for all $x\in X$. We say that $f(x,y)$ is {\it ($\nu,\alpha$)-strongly-convex-concave} when it is $\nu$-strongly convex in $x$ and $\alpha$-strongly concave in $y$, and just strongly-convex-concave when there exist $(\nu,\alpha)$ such that the function is  ($\nu,\alpha$)-strongly-convex-concave. For functions of three variables this is applied with $y=(p,q)$: strong concavity is required \emph{jointly} in the last two variables (separate strong concavity in each does not imply it, e.g.\ $-(p^2+q^2)/2+2pq$), and joint concavity is what the tracking analysis uses. We will deal frequently with functions of three variables $f:\mathbb{R}^{m_1}\times \mathbb{R}^{m_2}\times \mathbb{R}_{\geq 0}^{m_3}$ ($m_1$, $m_2$ and $m_3 \in \mathbb{N}$) that are convex in the first variable and jointly concave in the second and third variable. We say that $(\overline x, \overline y,\overline z)$ is a {\it saddle point} of the function $f$ when
$f(\overline x,y,z) \leq f(\overline x,\overline y,\overline z)
\leq f(x,\overline y,\overline z)$ for all $x$, $y$, and $z$ in
$\mathbb{R}^{m_1}\times \mathbb{R}^{m_2}\times
\mathbb{R}_{\geq 0}^{m_3}$.
\subsection{Projected primal--dual dynamics}
Given a function of three variables $f(x,p,q)$ we define the associated  primal-dual dynamics as 

\begin{equation}\label{primal_dual} \tag{Primal-Dual}
  \begin{aligned}
       \displaystyle\frac{dx}{d\tau} & = - \displaystyle\frac{\partial
      f}{\partial x}(x(\tau),p(\tau),q(\tau)),
    \\
    \displaystyle\frac{dp}{d\tau} & = \displaystyle\frac{\partial
      f}{\partial p}(x(\tau),p(\tau),q(\tau)),
    \\
    \displaystyle\frac{dq}{d\tau} & = [\displaystyle\frac{\partial
      f}{\partial q}(x(\tau),p(\tau),q(\tau))]^+_{q},
\end{aligned}
  \end{equation}
where the projection is applied componentwise and
\[
[b]^+_a = \begin{cases}
b, & a>0,\\
\max\{b,0\}, & a=0.
\end{cases}
\]
\section{Problem formulation and motivating questions}\label{sec:theoretical_formulation}
For each $t\geq0$, consider the constrained problem
\begin{align}
  \begin{split}\label{problem_original_time_dependent}
    \min\limits_x & \quad f_t(x),
    \\
    \textrm{s.t } & \quad h_t(x) = 0,
    \\
    & \quad g_t(x)\leq 0,
  \end{split}
\end{align}
where $f_t:\real^m\to\real$, $h_t:\real^m\to\real^{n_1}$, and
$g_t:\real^m\to\real^{n_2}$ are twice continuously differentiable in
$x$. We follow a continuous curve of minimizers $\hat x(t)$ and associated
multipliers $\hat p(t)$ and $\hat q(t)$ satisfying
\begin{equation}\label{kkt}
  \begin{array}{l}
    \nabla f_t(\hat{x}(t))  + \sum\limits_i^{n_1}\hat{p}^i(t)\nabla
    h^i_t(\hat{x}(t)) + \sum\limits_i^{n_2}\hat{q}^i(t)\nabla g^i_t(\hat{x}(t)) = 0, 
    \\
    \noalign{\medskip} 
    h_t(\hat{x}(t)) = 0,\ 
    g_t(\hat{x}(t)) \leq 0,\ 
    \hat{q}^i(t) \geq 0,\ 
    \hat{q}^i(t)g^i_t(\hat{x}(t)) = 0\textrm{ for }i\in\{1,\ldots,n_2\}.
  \end{array}
\end{equation}
The \emph{Linear Independence Constraint Qualification} (LICQ) holds at
$\hat x(t)$ when the equality-constraint gradients and the gradients of
the active inequalities are linearly independent. In the convex setting,
\emph{Slater's constraint qualification} holds when an equality-feasible
point $\overline x(t)$ satisfies $g_t^i(\overline x(t))<0$ for every
inequality. These conditions ensure the existence of KKT multipliers in the
settings where they are invoked; see \cite{DPB:99,ALD-RTR:14}.

\subsection{Why regularization is needed}
Our aim is to track $\hat x(t)$ with an explicit error bound while
remaining robust to bounded implementation errors. In the convex setting,
the connection between KKT points and Lagrangian saddle points suggests a
primal--dual approach. Using the time-dependent Lagrangian, define
\begin{equation}\label{lagrangian}
  \mathcal{L}^t(x,p,q) = f_t(x) + p^Th_t(x) + q^Tg_t(x).
\end{equation}
For a convex problem satisfying the relevant constraint qualification, a
KKT point is a saddle point of~\eqref{lagrangian}. This motivates the flow
\begin{align}
\label{primal_dual_lagrangian}
    \displaystyle\frac{dx}{d\tau} & = - \displaystyle\frac{\partial
      \mathcal{L}^t}{\partial x}(x(\tau),p(\tau),q(\tau)), \notag
    \\
    \displaystyle\frac{dp}{d\tau} & = \displaystyle\frac{\partial
      \mathcal{L}^t}{\partial p}(x(\tau),p(\tau),q(\tau)), \notag
    \\
    \displaystyle\frac{dq}{d\tau} & = [\displaystyle\frac{\partial
      \mathcal{L}^t}{\partial q}(x(\tau),p(\tau),q(\tau))]^+_{q},
    \end{align}
as a means to follow the solution. Without sufficient strong convexity and concavity, however,
the flow need not converge to a fixed equilibrium. Consider
\begin{align*}
\min\limits_{x} &\quad  x^1 + x^2, \\
\textrm{s.t.} & \quad x^1 + x^2 -1 = 0.
\end{align*}
Its Lagrangian is $\mathcal{L}=x^1+x^2+p(x^1+x^2-1)$.
Figure~\ref{fig:pd} shows persistent oscillations of the primal variables
under the corresponding saddle flow.
\begin{figure}[H]
  \centering
 \includegraphics[width=0.58\textwidth,keepaspectratio=true]{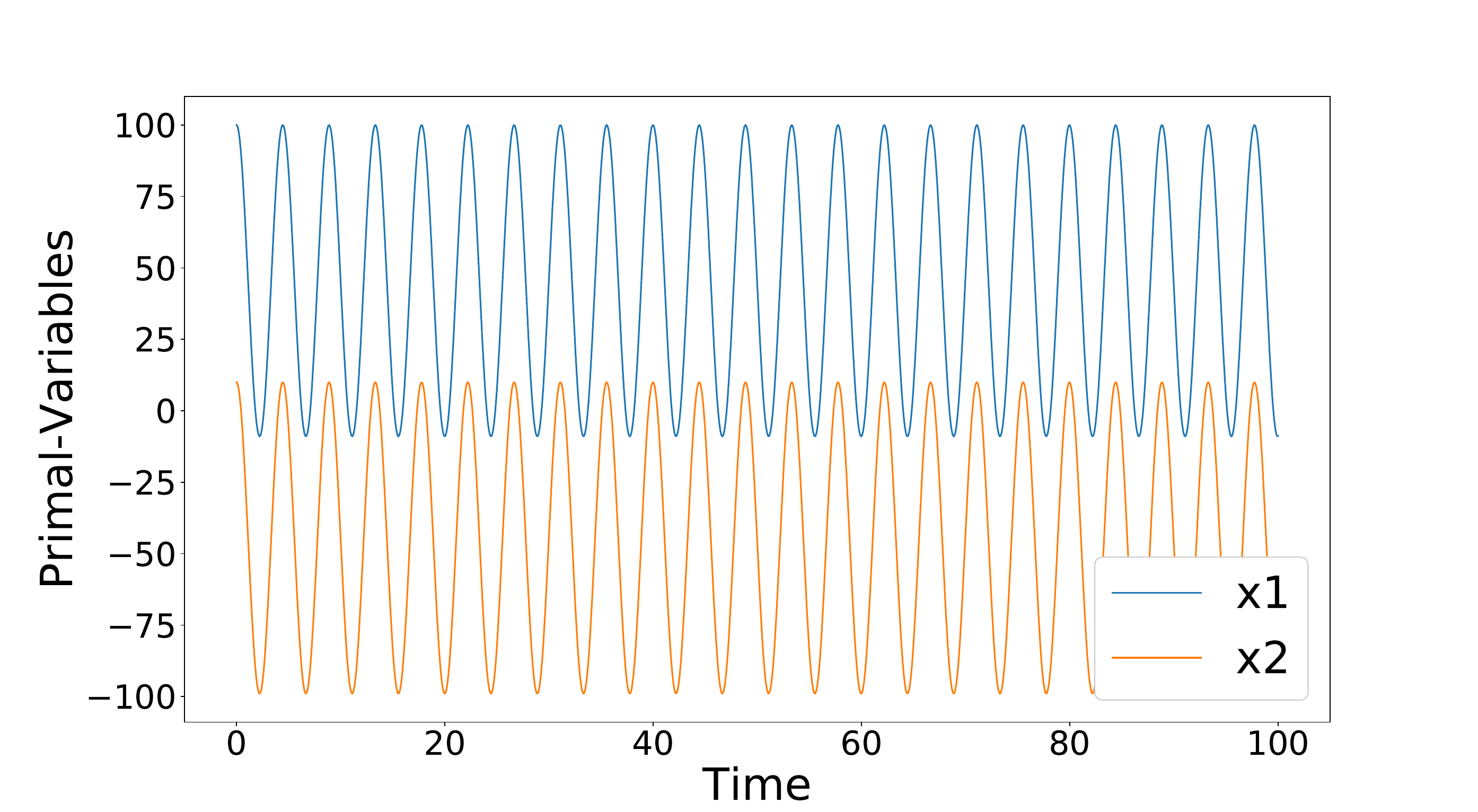}
\caption{Unregularized saddle flow for the flat constrained problem above.
The primal variables oscillate rather than converge to a fixed
primal--dual equilibrium.}\label{fig:pd}
\end{figure}

\section{Tracking and robustness for strongly convex--strongly concave saddle flows}\label{sec:tracking-robustness}
This section derives tracking and robustness bounds for the projected
primal--dual flow generated by a uniformly strongly convex--strongly
concave function $F$. The result is the analytical core used in both the
convex and nonconvex constructions.
\subsection{Tracking a moving saddle point}
\begin{theorem}[Bounds on primal-dual tracking of time-varying saddle points]\label{th:tracking}
For each $t\in\realnonneg$, let
$F^t:\real^m\times\real^{n_1}\times\realnonneg^{n_2}\to\real$ be
$\mathcal C^1$ and $(\nu,\alpha)$-strongly-convex-concave, with
$\nu,\alpha>0$ independent of $t$. Assume that its gradient is locally
Lipschitz in $(x,p,q)$, uniformly on compact time intervals, and depends
continuously on $t$. Let
$\hat z(t)=(\hat x(t),\hat p(t),\hat q(t))$ be its unique saddle point.
Assume that $\hat z$ is locally absolutely continuous and that
\[
  \norm{\dot{\hat z}(t)}\leq\delta
\]
for almost every $t$.  If $q(0)\geq0$, then every trajectory of
\eqref{primal_dual}, written as
$z(\tau)=(x(\tau),p(\tau),q(\tau))$, with $t=\kappa\tau$ and $\kappa>0$,
satisfies
\[
\begin{aligned}
&\norm{z(\tau)-\hat z(\kappa\tau)}\\
&\quad\leq e^{-m\tau}\norm{z(0)-\hat z(0)}
 +\frac{\kappa\delta}{m}\bigl(1-e^{-m\tau}\bigr),
\end{aligned}
\]
where $m=\min\{\nu,\alpha\}$. In particular,
\[
\limsup_{\tau\to\infty}
\norm{z(\tau)-\hat z(\kappa\tau)}
\leq\frac{\kappa\delta}{m}.
\]
\end{theorem}
\begin{proof}
Write $z=(x,p,q)$, $y=(p,q)$,
$\hat z=(\hat x,\hat p,\hat q)$, and $\hat y=(\hat p,\hat q)$,
all evaluated at $t=\kappa\tau$. On every compact time interval, the
unprojected vector field is continuous in time and locally Lipschitz in
$z$, with a uniform local Lipschitz constant.  The standard
projected-dynamical-system construction therefore gives a unique local
Carath\'eodory solution.  More explicitly, one may truncate the vector
  field outside a sufficiently large ball and apply the standard
time-dependent extension of
\cite[Proposition~$2.2$]{AC-EM-JC:16-scl} (obtained by the same Euler
polygonal construction on a compact time interval) to the resulting
globally Lipschitz field.  The estimate below then shows that the solution
does not reach the truncation boundary; compare the proof of
\cite[Lemma~$4.3$]{AC-EM-JC:16-scl}.  Continuation on successive compact
time intervals gives a global solution.  The nonnegative orthant is
forward invariant. Moreover, componentwise projection gives
\[
 \langle q-\hat q,[\nabla_qF^t]^+_q\rangle
 \leq \langle q-\hat q,\nabla_qF^t\rangle.
\]
For
$V(\tau)=\norm{z(\tau)-\hat z(\kappa\tau)}^2$, the chain rule and
$t=\kappa\tau$ give, at almost every $\tau$,
\begin{align*}
\frac12\dot V
&=\left\langle z-\hat z,
       \dot z-\kappa\dot{\hat z}(t)\right\rangle\\
&=-\left\langle x-\hat x,\nabla_xF^t(x,y)\right\rangle
  +\left\langle p-\hat p,\nabla_pF^t(x,y)\right\rangle\\
&\quad
  +\left\langle q-\hat q,
       [\nabla_qF^t(x,y)]^+_q\right\rangle
  +\left\langle\hat z-z,\kappa\dot{\hat z}(t)\right\rangle\\
&\leq-\left\langle x-\hat x,\nabla_xF^t(x,y)\right\rangle
  +\left\langle y-\hat y,\nabla_yF^t(x,y)\right\rangle
  +\left\langle\hat z-z,\kappa\dot{\hat z}(t)\right\rangle,
\end{align*}
where the last line uses the projection inequality above. Strong convexity
in $x$ and strong concavity in $y$, respectively, yield
\begin{align*}
-\left\langle x-\hat x,\nabla_xF^t(x,y)\right\rangle
&\leq F^t(\hat x,y)-F^t(x,y)
  -\frac\nu2\norm{x-\hat x}^2,\\
\left\langle y-\hat y,\nabla_yF^t(x,y)\right\rangle
&\leq F^t(x,y)-F^t(x,\hat y)
  -\frac\alpha2\norm{y-\hat y}^2.
\end{align*}
Substitution gives
\begin{align*}
\frac12\dot V
&\leq F^t(\hat x,y)-F^t(x,y)-\frac\nu2\norm{x-\hat x}^2\\
&\quad+F^t(x,y)-F^t(x,\hat y)
  -\frac\alpha2\norm{y-\hat y}^2
  +\langle \hat z-z,\kappa\dot{\hat z}(t)\rangle.
\end{align*}
The two function differences in this display satisfy
\begin{align*}
&F^t(\hat x,y)-F^t(x,y)+F^t(x,y)-F^t(x,\hat y)\\
&\quad=\bigl[F^t(\hat x,y)-F^t(\hat x,\hat y)\bigr]
 +\bigl[F^t(\hat x,\hat y)-F^t(x,\hat y)\bigr].
\end{align*}
Because $\hat y$ maximizes the $\alpha$-strongly concave
$F^t(\hat x,\cdot)$ over the convex set
$\real^{n_1}\times\realnonneg^{n_2}$ and $y(\tau)$ is feasible,
$F^t(\hat x,y)-F^t(\hat x,\hat y)\leq-\frac{\alpha}{2}\norm{y-\hat y}^2$;
likewise $\hat x$ is the unconstrained minimizer of the $\nu$-strongly
convex $F^t(\cdot,\hat y)$, so
$F^t(\hat x,\hat y)-F^t(x,\hat y)\leq-\frac{\nu}{2}\norm{x-\hat x}^2$.
Combining these two bounds with the preceding display and using
Cauchy--Schwarz and $\norm{\dot{\hat z}(t)}\leq\delta$ yields
\[
 \frac12\dot V
 \leq-\nu\norm{x-\hat x}^2-\alpha\norm{y-\hat y}^2
      +\kappa\delta\norm{z-\hat z}
 \leq-mV+\kappa\delta\sqrt V,
 \qquad m=\min\{\nu,\alpha\}.
\]
Set $r(\tau)=\sqrt{V(\tau)}=\norm{z(\tau)-\hat z(\kappa\tau)}$.
Whenever $r(\tau)>0$, division by $r(\tau)$ gives
\[
  \dot r(\tau)\leq-mr(\tau)+\kappa\delta.
\]
The same inequality holds in the upper-Dini-derivative sense when
$r(\tau)=0$. Multiplying by $e^{m\tau}$ and integrating from $0$ to
$\tau$ therefore yields
\[
 r(\tau)
 \leq e^{-m\tau}r(0)
   +\frac{\kappa\delta}{m}\bigl(1-e^{-m\tau}\bigr),
\]
which is the claimed estimate; taking the limit superior proves the final
statement.
\end{proof}
\begin{remark}[Interpretation]
 Tracking therefore
improves when the saddle curve moves more slowly (smaller $\delta$), the
algorithm runs faster relative to problem time (smaller $\kappa$), or the
curvature modulus $m$ is larger.
\end{remark}

\subsection{Additive disturbances}
We next characterize disturbances that enter the dynamics additively.
Formally, with \(t=\kappa\tau\), we consider the perturbed projected
primal--dual dynamics
\begin{align}
\begin{split}\label{primal_dual_dynamics_perturbed}
  \frac{dx}{d\tau}(\tau)
  &=
  -\nabla_x F^t\bigl(x(\tau),p(\tau),q(\tau)\bigr)
  + e_x(t),
  \\
  \frac{dp}{d\tau}(\tau)
  &=
  \nabla_p F^t\bigl(x(\tau),p(\tau),q(\tau)\bigr)
  + e_p(t),
  \\
  \frac{dq}{d\tau}(\tau)
  &=
  \left[
    \nabla_q F^t\bigl(x(\tau),p(\tau),q(\tau)\bigr)
    + e_q(t)
  \right]^+_{q(\tau)}.
\end{split}
\end{align}
Here, \(e_x\), \(e_p\), and \(e_q\) are measurable, locally essentially
bounded vector-valued functions representing additive disturbances in the
primal and dual dynamics. Together with the spatial regularity assumptions
on the vector field, these conditions place the system within the standard
existence theory for time-dependent projected dynamical systems and ensure
the existence of local Carathéodory solutions for every initial condition
satisfying \(q(0)\geq 0\).
This additive
disturbance model includes a structured approximation whenever the
difference between the approximate and exact vector fields is bounded.

\begin{proposition}[Robustness to additive disturbances]
\label{robustness}
Under the assumptions of Theorem~\ref{th:tracking}, suppose additionally
that the disturbances satisfy
\[
\bigl\|\bigl(e_x(t),e_p(t),e_q(t)\bigr)\bigr\|
\leq \delta_e
\]
for almost every \(t\). Let \(m=\min\{\nu,\alpha\}\), let
\(z(\tau)=(x(\tau),p(\tau),q(\tau))\), and define
\[
r(\tau)=\bigl\|z(\tau)-\hat z(\kappa\tau)\bigr\|.
\]
Then every trajectory of the perturbed projected primal--dual dynamics
satisfies
\[
r(\tau)
\leq
e^{-m\tau}r(0)
+\frac{\kappa\delta+\delta_e}{m}
 \bigl(1-e^{-m\tau}\bigr).
\]
Consequently,
\[
\limsup_{\tau\to\infty}r(\tau)
\leq
\frac{\kappa\delta+\delta_e}{m}.
\]
\end{proposition}

\begin{proof}
Let
\[
V(\tau)=\bigl\|z(\tau)-\hat z(\kappa\tau)\bigr\|^2.
\]
The projection inequality used in the proof of
Theorem~\ref{th:tracking} remains valid for the perturbed dual field:
\[
\left\langle
q-\hat q,
\left[\nabla_qF^t(z)+e_q(t)\right]^+_q
\right\rangle
\leq
\left\langle
q-\hat q,
\nabla_qF^t(z)+e_q(t)
\right\rangle.
\]
Therefore, the total contribution of the disturbances to
\(\tfrac12\dot V\) is bounded by
\[
\begin{aligned}
&\left\langle x-\hat x,e_x(t)\right\rangle
+\left\langle p-\hat p,e_p(t)\right\rangle
+\left\langle q-\hat q,e_q(t)\right\rangle\\
&\qquad\leq
\bigl\|z-\hat z(\kappa\tau)\bigr\|
\bigl\|\bigl(e_x(t),e_p(t),e_q(t)\bigr)\bigr\|
\leq \delta_e\sqrt{V}.
\end{aligned}
\]
Combining this estimate with the Lyapunov inequality from
Theorem~\ref{th:tracking} gives, for almost every \(\tau\),
\[
\frac12\dot V
\leq
-mV+(\kappa\delta+\delta_e)\sqrt V.
\]
Applying the same scalar comparison argument as in
Theorem~\ref{th:tracking} yields
\[
\sqrt{V(\tau)}
\leq
e^{-m\tau}\sqrt{V(0)}
+\frac{\kappa\delta+\delta_e}{m}
 \bigl(1-e^{-m\tau}\bigr),
\]
which proves both claims.
\end{proof}

\begin{remark}[Interpretation]
The motion of the saddle point and the implementation disturbances enter
the ultimate tracking bound additively. This decomposition makes it
possible to distinguish tracking error caused by the evolution of the
optimization problem from error caused by imperfect implementation of the
primal--dual dynamics.
\end{remark}
\section{Directional regularization for convex problems}\label{sec:regularization}
In this section, we adapt the results of the previous section to solve time-dependent convex optimization problems. That is, in problem~\eqref{problem_original_time_dependent}, we assume that $f_t(x)$ and $g_t(x)$ are convex and $h_t(x) = Ax - b_t$, where $A\in\mathbb R^{n_1\times m}$ has full row rank.

We start by noting that the results on tracking and robustness of
Section~\ref{sec:tracking-robustness} are not directly applicable to the time-varying optimization problem~\eqref{problem_original_time_dependent} by simply
taking $F^t = \mathcal{L}^t$. This is because the
Lagrangian~\eqref{lagrangian} is linear in the dual variables and
hence it fails to be strongly concave in these arguments. Regularization is a means to overcome this issue, but it produces a modification of the saddle-points of the Lagrangian. For instance, the standard Tikhonov regularization, see~\cite{JK-AN-UVS:11}, of $\Lc^t$ would
make the Lagrangian strongly-convex-concave by selecting $\nu$, $\alpha \in
\realpos$, and defining
\begin{align}\label{abstract_problem::augmented_lagrangian_classical}
  \hat{\mathcal{L}}^t_{\nu,\alpha}(x,p,q) = f_t(x) + p^T(Ax - b_t) +
  q^Tg_t(x) + \frac{\nu}{2}\norm{x}^2 - \frac{\alpha}{2}
  \big(\norm{p}^2 +\norm{q}^2 \big).
\end{align}
In general, this regularization perturbs the location of saddle points
without making any use of the underlying structure of the problem. Namely, the regularization term $\nu\norm{x}^2/2$ acts equally in all directions, and does not take advantage of the fact that adding
terms that satisfy the equality constraints are harmless to the solution of the optimization problem, leaving only some directions to be regularized (explicitly described below). Motivated by this observation, we instead propose a
geometry-aware regularization.
\subsection{Directional construction}
\begin{definition}[Augmented-Regularized Lagrangian] Given a Lagrangian $\mathcal{L}^t$  and  $\nu_1$, $\nu_2$, $\alpha\in\mathbb{R}_{\geq 0}$, we introduce the corresponding Augmented-Regularized Lagrangian as
\begin{align}\label{abstract_problem::augmented_lagrangian}
  \mathcal{L}^t_{\nu,\alpha}(x,p,q) & =f_t(x) + p^T(Ax - b_t) + q^Tg_t(x)
  \notag
  \\
  & \quad + \frac{\nu_1}{2}\norm{Ax - b_t}^2 + \frac{\nu_2}{2}
  \sum_{l=1}^{\dim (\ker(A))}(v_l^Tx)^2 - \frac{\alpha}{2} \big(\norm{p}^2
  +\norm{q}^2\big),
\end{align}
where $\{v_l \}$ is an orthonormal basis of the kernel of $A$.
\end{definition}
 We keep $A$ time-independent in this paper.  A time-varying extension
would require, in addition to constant rank, a locally absolutely
continuous projector onto $\ker A(t)$, a uniform lower bound on the
positive singular values of $A(t)$, and the temporal regularity required
by Theorem~\ref{th:tracking}.  Such an extension is not claimed here.
 We refer to the saddle points of this modified Lagrangian as $(\hat x_{\nu,\alpha},\hat p_{\nu,\alpha},\hat q_{\nu,\alpha})$.
Note that the addition of the term $\nu_1\norm{h_t(x)}^2/2$ does not
modify the location of the saddle points of the original problem,
while the addend $\nu_2 \sum_l(v_l^Tx)^2/2$ makes the augmented
Lagrangian strongly convex in the primal variables, as stated next.

\begin{proposition}[Strong convexity and concavity from directional regularization]\label{prop:subop-cons}
  For  $\nu_1$, $\nu_2$, $\alpha \in \realpos$, 
  \begin{align*}
    \nu = \lambda_{\min}(\nu_1A^TA +\nu_2\sum_l v_lv_l^T)
  \end{align*}
  is positive, and the
  Lagrangian~\eqref{abstract_problem::augmented_lagrangian} is
  $\nu$-strongly convex and $\alpha$-strongly concave on
  $\real^{m} \times \real^{n_1} \times \realnonneg^{n_2}$.
\end{proposition}
\begin{proof}
  The primal component of the Hessian of $\mathcal{L}^t_{\nu,\alpha}$ is
  \[
  \nabla_x^2\mathcal{L}^t_{\nu,\alpha}=\nabla^2f_t
  +\sum_iq^i\nabla^2g^i_t+\nu_1A^TA+\nu_2\sum_l v_lv_l^T.
  \]
  This is the sum of positive-semidefinite matrices (using
  convexity of $f_t$ and of each $g_t^i$ together with $q\geq0$).
  A straightforward argument shows that the eigenvalues of the matrix $\nu_1A^TA + \nu_2\sum
  v_lv_l^T$ are the union of the strictly positive eigenvalues of the matrix
  $\nu_1A^TA$ and the strictly positive eigenvalues of the matrix $\nu_2\sum
  v_lv_l^T$, and the result follows.
\end{proof}

\begin{remark}[Interpretation]
The row space of $A$ and the kernel of $A$ are complementary.
The $\nu_1$ term supplies convexity in the former directions, the $\nu_2$
term in the latter, and $\alpha$ supplies dual concavity. This geometric
separation is the key difference from Tikhonov regularization.

The directional construction is particularly suited when affine equalities
encode persistent conservation, balance, or coupling relations. In
network-flow models, $A$ may be a reduced
node--edge incidence matrix, with its kernel describing feasible
circulations. In resource-allocation and scheduling problems, the rows of
$A$ may encode total-resource or budget constraints, while its kernel
contains redistributions that preserve those totals. Similar structures
arise in power, water, communication, inventory, and transportation
networks. Redundant balance equations may be removed to obtain the
full-row-rank representation assumed here.
\end{remark}

\subsection{Regularization bias}

The regularization of the Lagrangian comes at the cost of perturbing
the location of the saddle points corresponding to the original
optimization problem. Our next results provide explicit bounds on the size of
the perturbation as a function of the parameters employed in the
regularization.  The results are analogous to those presented in \cite{JK-AN-UVS:11} for Tikhonov regularization. We will assume that the following assumptions hold throughout the rest of this section.
\begin{assumption}[Slater's condition]\label{assumption::slater}
For each $t$, there is a point $\overline{x}(t)$ such that
$A\overline{x}(t)-b_t=0$ and $g_t^i(\overline{x}(t))<0$ for every
$i\in\{1,\ldots,n_2\}$.
\end{assumption}
\begin{assumption}[Compactness of the feasible set]\label{assumption::comp}
The set
\[
X_t=\{x\in\mathbb R^m:Ax-b_t=0,\ g_t(x)\leq0\}
\]
is compact, and $D_t$ is chosen so that $\max_{x\in X_t}\norm{x}\leq D_t$.
\end{assumption}
\begin{assumption}[Lipschitz gradients]\label{assumption::grad_lips}
The gradients $\nabla f_t$ and $\nabla g_t^i$ are globally Lipschitz,
with constants $L_t^f$ and $L_t^g$, respectively, where the latter is
valid for every $i\in\{1,\ldots,n_2\}$.
\end{assumption}

Let $P=\sum_l v_lv_l^T$ and introduce the Lagrangian with primal
regularization only,
\[
\mathcal L_\nu^t=\mathcal L^t+\frac{\nu_1}{2}\norm{Ax-b_t}^2
+\frac{\nu_2}{2}x^TPx.
\]
Let $(\hat x_\nu,\hat p_\nu,\hat q_\nu)$ and
$(\hat x_{\nu,\alpha},\hat p_{\nu,\alpha},\hat q_{\nu,\alpha})$
denote saddle points of $\mathcal L_\nu^t$ and
$\mathcal L_{\nu,\alpha}^t$, respectively.

\begin{proposition}[Explicit perturbation bounds for augmented regularization]
\label{prop:perturbation-bounds}
Fix $t$ and define
\begin{align*}
\lambda_1&=\lambda_{\min}(\nu_1A^TA+\nu_2P),
&\lambda_2&=\lambda_{\max}(\nu_2P),\\
\gamma_t&=-\max_i g_t^i(\overline x(t))>0,
&M_t^f&=\max_{x\in X_t}\norm{\nabla f_t(x)},\\
&&M_t^g&=\max_{x\in X_t}\norm{\nabla g_t(x)}_2,
\end{align*}
where $\norm{\nabla g_t(x)}_2$ is the induced matrix norm. Let $d_t$ be
any known lower bound satisfying $d_t\leq f_t(\hat x_\nu(t))$
\greenfix{(see Remark~\ref{rem:computable-dt} for explicit computable
choices)}, and set
\begin{align*}
\mathsf Q_t&=
\frac{f_t(\overline x(t))+(\lambda_2/2)\norm{\overline x(t)}^2-d_t}
{\gamma_t},\\
\mathsf P_t&=\norm{(AA^T)^{-1}A}
\bigl(M_t^f+\lambda_2D_t+M_t^g\mathsf Q_t\bigr),\\
\rho_t&=\sqrt{\frac{\alpha(\mathsf P_t^2+\mathsf Q_t^2)}{4\lambda_1}}.
\end{align*}
Then
\begin{align}
\norm{\hat q_\nu(t)}&\leq\mathsf Q_t,
&\norm{\hat p_\nu(t)}&\leq\mathsf P_t,
\label{bounds-p-q}\\
\norm{\hat x_{\nu,\alpha}(t)-\hat x_\nu(t)}&\leq\rho_t.
\label{bound-primal-displacement}
\end{align}
Moreover,
\begin{align}
\left|f_t(\hat x_{\nu,\alpha}(t))-f_t(\hat x(t))\right|
&\leq M_t^f\rho_t+\frac{L_t^f}{2}\rho_t^2
+\frac{\lambda_2}{2}D_t^2,
\label{bound-objective}\\
\max\{0,g_t^i(\hat x_{\nu,\alpha}(t))\}
&\leq M_t^g\rho_t+\frac{L_t^g}{2}\rho_t^2,
\quad i=1,\ldots,n_2,
\label{bound-ineq-violation}\\
\norm{A\hat x_{\nu,\alpha}(t)-b_t}
&\leq\norm A\,\rho_t.
\label{bound-eq-violation}
\end{align}
\end{proposition}

\begin{proof}
All quantities below are evaluated at the fixed time $t$. Since the dual
variables of $\mathcal L_\nu^t$ are unregularized, $\hat x_\nu$ is
feasible for the original problem. The saddle inequalities and the
strictly feasible point $\overline x$ give
\begin{align*}
d_t
&\leq\mathcal L_\nu^t(\hat x_\nu,\hat p_\nu,\hat q_\nu)\\
&\leq\mathcal L_\nu^t(\overline x,\hat p_\nu,\hat q_\nu)\\
&\leq f_t(\overline x)-\gamma_t\norm{\hat q_\nu}_1
  +\frac{\lambda_2}{2}\norm{\overline x}^2.
\end{align*}
Rearrangement and $\norm{\cdot}\leq\norm{\cdot}_1$ prove the bound on
$\hat q_\nu$. Stationarity at $\hat x_\nu$, together with
$A\hat x_\nu-b_t=0$, gives
\[
\hat p_\nu=-(AA^T)^{-1}A
\bigl(\nabla f_t(\hat x_\nu)+\nu_2P\hat x_\nu
      +\nabla g_t(\hat x_\nu)\hat q_\nu\bigr).
\]
Compactness of $X_t$ and the bound on $\hat q_\nu$ now prove the bound
on $\hat p_\nu$.

For the displacement estimate, write $y=(p,q)$.
At fixed $\hat y_{\nu,\alpha}$, the dual penalty is independent of $x$:
\[
 \mathcal L_{\nu,\alpha}^t(x,\hat y_{\nu,\alpha})
 =\mathcal L_\nu^t(x,\hat y_{\nu,\alpha})
 -\frac{\alpha}{2}\norm{\hat y_{\nu,\alpha}}^2.
\]
The primal saddle inequality for $\mathcal L_{\nu,\alpha}^t$ therefore
remains unchanged after cancelling this common constant, and hence
\[
 \hat x_{\nu,\alpha}
 =\operatorname*{argmin}_x
   \mathcal L_\nu^t(x,\hat y_{\nu,\alpha}).
\]

Since $\mathcal L_\nu^t$ is $\lambda_1$-strongly convex in $x$, the
primal minimizing conditions at the two saddle points give
\begin{align*}
\mathcal L_\nu^t(\hat x_\nu,\hat y_{\nu,\alpha})
&\geq \mathcal L_\nu^t(\hat x_{\nu,\alpha},\hat y_{\nu,\alpha})
 +\frac{\lambda_1}{2}
  \norm{\hat x_{\nu,\alpha}-\hat x_\nu}^2,\\
\mathcal L_\nu^t(\hat x_{\nu,\alpha},\hat y_\nu)
&\geq \mathcal L_\nu^t(\hat x_\nu,\hat y_\nu)
 +\frac{\lambda_1}{2}
  \norm{\hat x_{\nu,\alpha}-\hat x_\nu}^2.
\end{align*}
Adding these inequalities and using that $\mathcal L_\nu^t$ is affine
in $y$ yields
\begin{align*}
\lambda_1\norm{\hat x_{\nu,\alpha}-\hat x_\nu}^2
&\leq
\left\langle
\nabla_y\mathcal L_\nu^t(\hat x_{\nu,\alpha},\hat y_{\nu,\alpha})
-\nabla_y\mathcal L_\nu^t(\hat x_\nu,\hat y_\nu),
\hat y_\nu-\hat y_{\nu,\alpha}
\right\rangle.
\end{align*}
The dual maximizing conditions, evaluated at the other saddle point,
are
\begin{align*}
\left\langle
\nabla_y\mathcal L_\nu^t(\hat x_\nu,\hat y_\nu),
\hat y_{\nu,\alpha}-\hat y_\nu
\right\rangle&\leq0,\\
\left\langle
\nabla_y\mathcal L_\nu^t(\hat x_{\nu,\alpha},\hat y_{\nu,\alpha})
-\alpha\hat y_{\nu,\alpha},
\hat y_\nu-\hat y_{\nu,\alpha}
\right\rangle&\leq0.
\end{align*}
Both comparison points belong to
$\mathbb R^{n_1}\times\mathbb R_{\geq0}^{n_2}$. Combining the last three
inequalities therefore gives
\begin{align*}
\lambda_1\norm{\hat x_{\nu,\alpha}-\hat x_\nu}^2
&\leq\alpha\bigl(\langle\hat y_{\nu,\alpha},\hat y_\nu\rangle
                 -\norm{\hat y_{\nu,\alpha}}^2\bigr)\\
&\leq\frac\alpha4\norm{\hat y_\nu}^2
\leq\frac\alpha4(\mathsf P_t^2+\mathsf Q_t^2),
\end{align*}
which proves \eqref{bound-primal-displacement}.

The Lipschitz-gradient Taylor estimate and the Cauchy--Schwarz inequality
give
\begin{align*}
|f_t(\hat x_{\nu,\alpha})-f_t(\hat x_\nu)|
&\leq \bigl|\bigl\langle\nabla f_t(\hat x_\nu),
      \hat x_{\nu,\alpha}-\hat x_\nu\bigr\rangle\bigr|
      +\frac{L_t^f}{2}
       \norm{\hat x_{\nu,\alpha}-\hat x_\nu}^2\\
&\leq \norm{\nabla f_t(\hat x_\nu)}
       \norm{\hat x_{\nu,\alpha}-\hat x_\nu}
      +\frac{L_t^f}{2}
       \norm{\hat x_{\nu,\alpha}-\hat x_\nu}^2\\
&\leq M_t^f\rho_t+\frac{L_t^f}{2}\rho_t^2.
\end{align*}
The last inequality uses $\hat x_\nu\in X_t$, and hence
$\norm{\nabla f_t(\hat x_\nu)}\leq M_t^f$, together with
\eqref{bound-primal-displacement}.

Because both $\hat x$ and $\hat x_\nu$ are feasible and $\hat x$
minimizes $f_t$ on $X_t$,
\[
0\leq f_t(\hat x_\nu)-f_t(\hat x)
\leq\frac{\nu_2}{2}\hat x^TP\hat x
\leq\frac{\lambda_2}{2}D_t^2.
\]
This proves \eqref{bound-objective}. Applying the same Taylor estimate
to each $g_t^i$, using $g_t^i(\hat x_\nu)\leq0$, proves
\eqref{bound-ineq-violation}. Finally,
$A\hat x_\nu=b_t$ and \eqref{bound-primal-displacement} give
\eqref{bound-eq-violation}.
\end{proof}

\begin{remark}[Computing the lower bound $d_t$]
\label{rem:computable-dt}
The lower bound required in Proposition~\ref{prop:perturbation-bounds} is
directly obtainable from the quantities already present in
Assumption~\ref{assumption::comp}. Indeed, for any reference point
$x_t^0\in\mathbb R^m$, convexity of $f_t$ and
$\norm{\hat x_\nu(t)}\leq D_t$ give
\[
 f_t(\hat x_\nu(t))
 \geq f_t(x_t^0)-\left\langle\nabla f_t(x_t^0),x_t^0\right\rangle
      -D_t\norm{\nabla f_t(x_t^0)}.
\]
Consequently, one may use the explicit choice
\[
 d_t=f_t(x_t^0)-\left\langle\nabla f_t(x_t^0),x_t^0\right\rangle
      -D_t\norm{\nabla f_t(x_t^0)}.
\]
In particular, taking $x_t^0=0$ yields
$d_t=f_t(0)-D_t\norm{\nabla f_t(0)}$; if $f_t$ is known to be
nonnegative, one may simply take $d_t=0$. A tighter computable alternative
is $d_t=\min_{\norm{x}\leq D_t}f_t(x)$, which is a convex problem over a
ball and does not involve the original constraints.
\end{remark}

The next corollary combines the convexity, concavity, perturbation, and tracking
results.

\begin{corollary}[Tracking and performance with the augmented-regularized Lagrangian]
\label{cor:convex-tracking}
Suppose that $f_t$ and $g_t^i$ are convex and twice continuously
differentiable, $h_t(x)=Ax-b_t$, and
Assumptions~\ref{assumption::slater}--\ref{assumption::grad_lips} hold.
Fix $\nu_1,\nu_2,\alpha>0$ and consider the projected primal--dual
dynamics initialized with $q(0)\geq0$. In addition, suppose that the saddle curve
of $\mathcal L_{\nu,\alpha}^t$ and its vector field satisfy all temporal
regularity and speed assumptions in Theorem~\ref{th:tracking}. Then:
\begin{itemize}
\item[(i)] The primal-dual dynamics tracks the saddle curve with the
guarantees of Theorem~\ref{th:tracking}, and with the robustness guarantee
of Proposition~\ref{robustness} when bounded disturbances are present,
where $\lambda_{\min}(\nu_1A^TA+\nu_2P)$ is a primal strong-convexity
modulus and $\alpha$ is the joint strong-concavity modulus in $(p,q)$.
Consequently, the contraction modulus in the tracking and robustness bounds
is
\[
m=\min\bigl\{\lambda_{\min}(\nu_1A^TA+\nu_2P),\alpha\bigr\}.
\]
\item[(ii)] Let $t=\kappa\tau$ and define the primal tracking error
\[
e_x(\tau)=\norm{x(\tau)-\hat x_{\nu,\alpha}(t)}.
\]
With $z(\tau)=(x(\tau),p(\tau),q(\tau))$, define the projected
primal--dual residual
\begin{align*}
r_{\rm pd}(\tau)=\Big(&
\norm{\nabla_x\mathcal L_{\nu,\alpha}^t(z(\tau))}^2
+\norm{\nabla_p\mathcal L_{\nu,\alpha}^t(z(\tau))}^2\\
&+\norm{[\nabla_q\mathcal L_{\nu,\alpha}^t(z(\tau))]^+_{q(\tau)}}^2
\Big)^{1/2}.
\end{align*}
In addition to the primal--dual estimate in part~(i), the tracking error
always satisfies the computable pointwise bound
\begin{equation}
e_x(\tau)
\leq\norm{z(\tau)-\hat z_{\nu,\alpha}(t)}
\leq
\greenfix{\frac{r_{\rm pd}(\tau)}{
\min\{\lambda_{\min}(\nu_1A^TA+\nu_2P),\alpha\}}}
=:\epsilon_x^{\rm res}(\tau).
\label{bound-residual-tracking}
\end{equation}
The trajectory itself satisfies
\begin{align*}
\left|f_t(x(\tau))-f_t(\hat x(t))\right|
&\leq M_t^f\bigl(\rho_t+e_x(\tau)\bigr)
 +\frac{L_t^f}{2}\bigl(\rho_t+e_x(\tau)\bigr)^2
 +\frac{\lambda_2}{2}D_t^2,\\
\max\{0,g_t^i(x(\tau))\}
&\leq M_t^g\bigl(\rho_t+e_x(\tau)\bigr)
 +\frac{L_t^g}{2}\bigl(\rho_t+e_x(\tau)\bigr)^2,
 \quad i=1,\ldots,n_2,\\
\norm{Ax(\tau)-b_t}
&\leq\norm A\bigl(\rho_t+e_x(\tau)\bigr).
\end{align*}
\end{itemize}
\end{corollary}
\begin{proof}
Part~(i) follows from Proposition~\ref{prop:subop-cons},
Theorem~\ref{th:tracking}, and Proposition~\ref{robustness}.  For part~(ii),
write $y=(p,q)$ and
$\hat y=(\hat p_{\nu,\alpha},\hat q_{\nu,\alpha})$, all at the fixed
time $t$.  Let
\[
v_t(z)=
\begin{bmatrix}
-\nabla_x\mathcal L_{\nu,\alpha}^t(z)\\
\nabla_p\mathcal L_{\nu,\alpha}^t(z)\\
[\nabla_q\mathcal L_{\nu,\alpha}^t(z)]^+_q
\end{bmatrix},
\]
which is the right-hand side of the projected primal--dual dynamics.  The
componentwise projection satisfies the same inequality used in the proof
of Theorem~\ref{th:tracking}, namely
\[
\left\langle q-\hat q_{\nu,\alpha},
[\nabla_q\mathcal L_{\nu,\alpha}^t(z)]^+_q\right\rangle
\leq
\left\langle q-\hat q_{\nu,\alpha},
\nabla_q\mathcal L_{\nu,\alpha}^t(z)\right\rangle.
\]
Strong convexity in $x$, strong concavity in $y$, and the saddle
inequalities therefore give, exactly as in that proof,
\begin{align*}
\left\langle z-\hat z_{\nu,\alpha}(t),v_t(z)\right\rangle
&\leq \mathcal L_{\nu,\alpha}^t(\hat x_{\nu,\alpha},y)
       -\mathcal L_{\nu,\alpha}^t(x,y)
-\frac{\lambda_1}{2}\norm{x-\hat x_{\nu,\alpha}}^2\\
&\quad+\mathcal L_{\nu,\alpha}^t(x,y)
       -\mathcal L_{\nu,\alpha}^t(x,\hat y)
-\frac{\alpha}{2}\norm{y-\hat y}^2.
\end{align*}
Adding and subtracting
$\mathcal L_{\nu,\alpha}^t(\hat x_{\nu,\alpha},\hat y)$, the strong saddle
inequalities give the additional  estimate
\begin{align*}
\mathcal L_{\nu,\alpha}^t(\hat x_{\nu,\alpha},y)
-\mathcal L_{\nu,\alpha}^t(x,\hat y)
&=\mathcal L_{\nu,\alpha}^t(\hat x_{\nu,\alpha},y)
  -\mathcal L_{\nu,\alpha}^t(\hat x_{\nu,\alpha},\hat y)\\
&\quad+\mathcal L_{\nu,\alpha}^t(\hat x_{\nu,\alpha},\hat y)
  -\mathcal L_{\nu,\alpha}^t(x,\hat y)\\
&\leq-\frac{\alpha}{2}\norm{y-\hat y}^2
       -\frac{\lambda_1}{2}
        \norm{x-\hat x_{\nu,\alpha}}^2.
\end{align*}
Consequently,
\[
\left\langle z-\hat z_{\nu,\alpha}(t),v_t(z)\right\rangle
\leq-\lambda_1\norm{x-\hat x_{\nu,\alpha}}^2
       -\alpha\norm{y-\hat y}^2
\leq-\min\{\lambda_1,\alpha\}
       \norm{z-\hat z_{\nu,\alpha}(t)}^2.
\]
where $\lambda_1=\lambda_{\min}(\nu_1A^TA+\nu_2P)$.  On the other hand,
Cauchy--Schwarz gives
\[
-\left\langle z-\hat z_{\nu,\alpha}(t),v_t(z)\right\rangle
\leq\norm{z-\hat z_{\nu,\alpha}(t)}\norm{v_t(z)}.
\]
Since $\norm{v_t(z(\tau))}=r_{\rm pd}(\tau)$, cancellation (with the
zero-distance case being immediate) proves
\eqref{bound-residual-tracking}.  Next,
the triangle inequality and \eqref{bound-primal-displacement} give
\[
\norm{x(\tau)-\hat x_\nu(t)}\leq e_x(\tau)+\rho_t.
\]
Apply the same Lipschitz-gradient Taylor estimates used in the proof of
Proposition~\ref{prop:perturbation-bounds}.  Equality feasibility follows from
$A\hat x_\nu(t)=b_t$.  The remaining saddle-point bounds are precisely
those of Proposition~\ref{prop:perturbation-bounds}.  Finally, the
right-hand sides are nondecreasing functions of $e_x(\tau)\geq0$, so
\eqref{bound-residual-tracking} permits its replacement by
$\epsilon_x^{\rm res}(\tau)$.
\end{proof}
\begin{remark}[Interpretation]
The conclusion separates dynamic tracking error from static
regularization bias: Theorem~\ref{th:tracking} controls the former, while
Proposition~\ref{prop:perturbation-bounds} controls objective and
feasibility errors relative to the original problem. 
\end{remark}

\section{Local convexification for nonconvex problems}\label{sec:non_convex_setting}

We now allow $f_t$, $h_t$, and $g_t$ in
\eqref{problem_original_time_dependent} to be nonconvex. The goal is local:
we follow a regular KKT trajectory inside a neighborhood. A direct extension of
Section~\ref{sec:regularization} is not enough because KKT points of a
nonconvex problem need not be global saddle points. The
augmentation term has been used to convexify optimization problems
previously. See, for instance, \cite{DPB:75-siam}, where the term $\frac{c}{2}\norm{h_t(x)}^2$ is added to the Lagrangian in order to obtain a locally
convex Lagrangian and convex optimization algorithms can be applied
afterward. Although this approach works well in terms of
convexification, it is not enough for our tracking purposes, which
should enjoy stronger convergence properties. Another issue of maintaining the same approach as in Section~\ref{sec:regularization} is that
the computation of the kernel $\nabla h_t(x)$ at any point can be very
demanding. In
order to overcome the technical difficulties of the problem we resort
to proximal regularization, along the lines of
\cite{DPB:79,ALD-RTR:14,SMR:91}. The basic idea is to replace the
nonconvex problem locally by a family of strongly convex subproblems.

\subsection{Local model and assumptions}
The following assumptions  hold throughout this section.

\begin{assumption}[Local KKT regularity and second-order sufficiency]\label{assumption::seven}
We assume that LICQ holds along the curve of minimizers of problem~\eqref{problem_original_time_dependent}, denoted by $\hat{x}(t)$. Also, at the minimizers $\hat{x}(t)$ (and the corresponding unique Lagrange
    multipliers $\hat{p}(t)$ and $\hat{q}(t)$), the inequality 
    \[
    z^T\nabla_{xx}\mathcal{L}^t(\hat{x}(t),\hat{p}(t),\hat{q}(t))z >
    0
    \]
    holds for every nonzero $z\in\mathbb{R}^m$ such that $\nabla
      h_t^i(\hat{x}(t))^Tz = 0$ for $i\in\{1,\ldots,n_1\}$ and $
      \nabla g_t^j(\hat{x}(t))^Tz = 0$ for each $j$ such that
      $\hat{q}^j(t) > 0$. This is a strong second-order sufficient condition.
\end{assumption}

\begin{assumption}[Temporal regularity]\label{assumption::time-regularity}
Fix a compact interval $[0,t_f]$. In a neighborhood of
$\{(t,\hat x(t)):t\in[0,t_f]\}$, the functions $f_t$, $h_t$, and $g_t$,
together with their first and second derivatives with respect to $x$, are
continuous in $(t,x)$ and locally Lipschitz in $t$, uniformly with respect
to $x$ on compact subsets. The KKT curve
$t\mapsto(\hat x(t),\hat p(t),\hat q(t))$ is continuous on $[0,t_f]$.
\end{assumption}

\subsection{Local envelope and implicit solution map}

\begin{definition}[Local convexification envelope]
For $c>0$, define $\varphi_c^t(\ell)$ locally by

\begin{equation}\label{def-varphi}
\begin{aligned}
\varphi_c^t(\ell)=\min_x\quad
  &f_t(x)+\frac{1}{2c}\norm{x-\ell}^2\\
\text{s.t.}\quad &h_t(x)=0,\qquad g_t(x)\leq0.
\end{aligned}
\end{equation}
The definition is restricted to $\ell$ and local minimizers near
$\hat x(t)$. We call~\eqref{def-varphi} the $(t,\ell)$-subproblem and
write $x_c(t,\ell)$ for its locally selected minimizer.
\end{definition}
The Lagrangian of this subproblem is
\begin{equation}\label{ell-lagrangian}
  \mathcal{L}^{t,\ell}_c(x,p,q)
  =f_t(x)+p^Th_t(x)+q^Tg_t(x)
   +\frac{1}{2c}\norm{\ell-x}^2,
\end{equation}
and its Hessian with respect to the primal variable is
\begin{equation}\label{hessian-lagrangian}
  \nabla_{xx}\mathcal{L}^{t,\ell}_c(x,p,q)
  =\nabla^2f_t(x)+\sum_i p^i\nabla^2h_t^i(x)
   +\sum_i q^i\nabla^2g_t^i(x)+\frac{I}{c}.
\end{equation}
Locally, and for $c>0$ sufficiently small, the term $I/c$ dominates
 and makes the Lagrangian strongly convex in $x$.

The point $\hat x(t)$ is a local minimizer of
$\ell\mapsto\varphi_c^t(\ell)$. The construction is motivated by two
additional properties established below: the envelope has a computable
gradient and is locally strongly convex. We first isolate the
unregularized gradient formula in
Proposition~\ref{prop::gradient_computation};
Proposition~\ref{lemma::gradient_varphi} extends it to the
dual-regularized envelope.

\begin{proposition}[Computation of  $\nabla\varphi_c^t$]\label{prop::gradient_computation}
Under Assumptions~\ref{assumption::seven} and
\ref{assumption::time-regularity}, and for $c>0$ sufficiently small, the
local envelope $\varphi_c^t$ is $\mathcal C^1$ near $\hat x(t)$ and
\[
\nabla \varphi_c^t(\ell)=\frac{\ell-x_c(t,\ell)}{c}.
\]
\end{proposition}
\begin{proof}
Fix $t$. At the reference solution, LICQ and
Assumption~\ref{assumption::seven}, together with the added Hessian term
$I/c$ in \eqref{hessian-lagrangian}, give the strong second-order
condition for the $(t,\ell)$-subproblem. The parametric KKT result
\cite[Theorem~$2$G.$9$]{ALD-RTR:14} therefore provides locally unique,
locally Lipschitz selections $(x_c,p_c,q_c)$ as functions of $\ell$. At a KKT point,
complementarity permits writing its local value as
\[
\varphi_c^t(\ell)=\mathcal L_c^{t,\ell}(x_c,p_c,q_c).
\]
At every point where this Lipschitz selection is differentiable, the
chain rule and primal stationarity cancel the derivative multiplying
$D_\ell x_c$.  The equality residual multiplying $D_\ell p_c$ is zero.
For an inequality with $q_c^i>0$, its residual is zero; for a strictly
inactive inequality, $q_c^i$ is locally zero; and at a weakly active
switching point both $q_c^i$ and its residual are zero.  Consequently,
\[
\nabla\varphi_c^t(\ell)=\frac{\ell-x_c(t,\ell)}c
\]
at almost every $\ell$.  The right-hand side is continuous.  Since the
envelope is locally Lipschitz, its Clarke generalized gradient at every
remaining point is therefore the singleton containing this vector.
By \cite[Proposition~$2.2.4$]{FHC:83}, the envelope is strictly
differentiable there.  The continuous gradient formula proves the
$\mathcal C^1$ claim.
\end{proof}

The identity above suggests an outer gradient flow
$\dot\ell=-\nabla\varphi_c^t(\ell)$. Computing its gradient exactly would
require solving a new $(t,\ell)$-subproblem at every instant. Instead, we
run an inner primal--dual flow whose primal state approximates
$x_c(t,\ell)$, and use that state to form an inexact outer direction. To
obtain a contracting inner saddle flow, we regularize only the dual
variables and define
\[
\mathcal{L}^{t,\ell}_{c,\alpha}(x,p,q) = f_t(x) + p^Th_t(x) +
q^Tg_t(x) + \displaystyle\frac{1}{2c}\norm{\ell - x}^2
-\displaystyle\frac{\alpha^2}{2}\norm{p}^2
-\displaystyle\frac{\alpha^2}{2}\norm{q}^2.
\]
The square in $\alpha^2$ makes the dependence on the perturbation parameter
symmetric around zero while providing concavity $\alpha^2$.

We base our constructions on the application of a recent generalization of the Implicit Function Theorem  (IFT) that includes generalized equations, see~\cite{ALD-RTR:14}. An
important feature, in contrast to the standard IFT, is that the
 obtained functions typically lack differentiability but often exhibit
 Lipschitz continuity. Particularly illuminating are the references
\cite{ALD-RTR:14,JG-DN-FG:01}.  For the sake of clarity of the
exposition, we define
\[
D(x,p,q,t,\ell,\alpha) =
\left(\displaystyle\frac{\partial\mathcal{L}^{t,\ell}_{c,\alpha}}{\partial
    x}(x,p,q),-\displaystyle\frac{\partial\mathcal{L}^{t,\ell}_{c,\alpha}}{\partial
    p}(x,p,q),-\displaystyle\frac{\partial\mathcal{L}^{t,\ell}_{c,\alpha}}{\partial
    q}(x,p,q)\right).
\]
Assumption~\ref{assumption::time-regularity} implies that $D$ and its
linearization with respect to $(x,p,q)$ are locally Lipschitz in $t$,
uniformly on compact subsets of the remaining variables.
Set
\[
C=\mathbb R^m\times\mathbb R^{n_1}\times
  \mathbb R_{\geq0}^{n_2}.
\]
\begingroup
\medskip
\noindent\textit{Projection and normal-cone formulation.}
For $q\in\mathbb R_{\geq0}^{n_2}$, the projected dual equilibrium
condition is equivalently
\[
\bigl[\nabla_q\mathcal L_{c,\alpha}^{t,\ell}(x,p,q)\bigr]^+_q=0
\quad\Longleftrightarrow\quad
0\in-\nabla_q\mathcal L_{c,\alpha}^{t,\ell}(x,p,q)
  +N_{\mathbb R_{\geq0}^{n_2}}(q).
\]
More generally,
\[
\nabla_q\mathcal L_{c,\alpha}^{t,\ell}(x,p,q)
-\bigl[\nabla_q\mathcal L_{c,\alpha}^{t,\ell}(x,p,q)\bigr]^+_q
\in N_{\mathbb R_{\geq0}^{n_2}}(q).
\]
Consequently, the inclusion
$D(x,p,q,t,\ell,\alpha)+N_C(x,p,q)\ni0$ used below is precisely the
normal-cone representation of the equilibrium conditions of the projected
primal--dual dynamics.
\medskip
\endgroup
\begin{proposition}[Local implicit solution map]
\label{prop:implicit-solution-map}

  Consider a continuous curve of local minimizers $\hat{x}(t)$ and its Lagrange multipliers, $\hat p(t)$ and $\hat q(t)$,
  satisfying Assumptions~\ref{assumption::seven} and
  \ref{assumption::time-regularity}.
  For $c>0$ sufficiently small, the variational inclusion
  \begin{equation}\label{equation::variational}
    D(x,p,q,t,\ell,\alpha) + N_C(x,p,q) \ni 0 
  \end{equation}
  admits a locally unique solution near
  $(\hat x(t),\hat p(t),\hat q(t))$. This solution is locally Lipschitz in
  the parameters $(t,\ell,\alpha)$ near $(t,\hat x(t),0)$ and is denoted
  by $\inv{x}(t,\ell,\alpha),\inv{p}(t,\ell,\alpha)$, and
  $\inv{q}(t,\ell,\alpha)$.
\end{proposition}
\begin{proof}
Fix $t_0\in[0,t_f]$ and set
\[
u_0=(\hat x(t_0),\hat p(t_0),\hat q(t_0)),
\qquad \theta_0=(t_0,\hat x(t_0),0).
\]
At $(u_0,\theta_0)$, inclusion~\eqref{equation::variational} is the KKT
generalized equation for the proximal subproblem.  Its Lagrangian Hessian
in the primal variable is
\[
\nabla_{xx}^2\mathcal L^{t_0}(u_0)+\frac Ic.
\]
LICQ and Assumption~\ref{assumption::seven} therefore satisfy the strong
second-order condition in
\cite[Theorem~$2$G.$9$]{ALD-RTR:14}; equivalently, the linearized KKT
generalized equation has a Lipschitz single-valued inverse, as characterized
in \cite[Theorem~$2$G.$8$]{ALD-RTR:14}.  Assumption~
\ref{assumption::time-regularity} then permits applying the generalized
Implicit Function Theorem \cite[Theorem~$2$B.$7$]{ALD-RTR:14}, which yields
a locally unique solution $u^{\rm imp}(t,\ell,\alpha)$, locally Lipschitz in
its parameters.

Because the reference solution graph is compact and the second-order and
LICQ conditions are open under the stated continuity assumptions, a
finite-cover argument permits common choices of $c$, parameter radii, and
Lipschitz constants along $[0,t_f]$, after shrinking the neighborhoods if
necessary.  Local uniqueness makes the selections agree on overlaps.
This constructs the asserted implicit map along the full compact curve.
\end{proof}
The mappings $\inv{x}$, $\inv{p}$, and $\inv{q}$ describe the locally
selected saddle point of $\mathcal{L}^{t,\ell}_{c,\alpha}$ when
$(t,\ell,\alpha)$ are treated as parameters. For compactness, write
\[
 z^{\rm imp}(t,\ell,\alpha)
 :=\bigl(\inv{x}(t,\ell,\alpha),\inv{p}(t,\ell,\alpha),
          \inv{q}(t,\ell,\alpha)\bigr).
\]
The positive design parameters $c$ and $\alpha$ will be selected inside
the uniform neighborhoods constructed below.

\begin{definition}[Dual-regularized local envelope]
For sufficiently small $\alpha>0$, define $\varphi^t_{c,\alpha}$ in a
neighborhood of $\hat x(t)$ by
  \begin{align}
    \varphi^t_{c,\alpha}(\ell)
    &= f_t(\inv{x}(t,\ell,\alpha))
     + \inv{p}(t,\ell,\alpha)^Th_t(\inv{x}(t,\ell,\alpha))
       \label{equation::varphi}\\
    &\quad + \inv{q}(t,\ell,\alpha)^Tg_t(\inv{x}(t,\ell,\alpha))
     + \frac{1}{2c}\norm{\inv{x}(t,\ell,\alpha)-\ell}^2\nonumber\\
    &\quad - \frac{\alpha^2}{2}\norm{\inv{p}(t,\ell,\alpha)}^2
     - \frac{\alpha^2}{2}\norm{\inv{q}(t,\ell,\alpha)}^2.\nonumber
  \end{align}
\end{definition}
\begin{remark}
Note that $\varphi^t_{c,0}(\ell) = \varphi_{c}^t(\ell)$.  \oprocend
\end{remark}

\begin{proposition}[Local envelope regularity]\label{lemma::gradient_varphi}
Under Assumptions~\ref{assumption::seven} and
\ref{assumption::time-regularity}, and for $c>0$ sufficiently small,
\greenfix{there exist $\bar\alpha>0$, $\sigma>0$, and a neighborhood
radius independent of $t$ and $\alpha$ such that, for every
$t\in[0,t_f]$ and $0\leq\alpha\leq\bar\alpha$,}
$\varphi^t_{c,\alpha}$ satisfies the following properties in that
neighborhood of $\hat x(t)$.
\begin{enumerate}
\item  $\varphi^t_{c,\alpha}$ is $\mathcal{C}^1$ and verifies $ \nabla
  \varphi^t_{c,\alpha}(\ell) = \displaystyle\frac{\ell - \inv{x}(t,\ell,\alpha)}{c}$.  

\item  \greenfix{$\varphi^t_{c,\alpha}$ is $\sigma$-strongly convex in
  this common neighborhood} (and hence strictly convex there).

\item  $\varphi^t_{c,\alpha}$ has a locally unique critical point
  $\hat x_\alpha(t)$ close to $\hat x_0(t)=\hat x(t)$.  It depends
  locally Lipschitz-continuously on $\alpha$ and satisfies
  $\hat x_\alpha(t)=\inv{x}(t,\hat x_\alpha(t),\alpha)$.
\end{enumerate}
\end{proposition}

\begin{proof}
We separate the argument according to the three conclusions. The
active-set and block-matrix details are given in
Appendix~\ref{app:envelope-derivation}.
Fix $t$ and, for each admissible $\alpha$, abbreviate
$(\inv{x},\inv{p},\inv{q})
=(\inv{x}(t,\ell,\alpha),\inv{p}(t,\ell,\alpha),
  \inv{q}(t,\ell,\alpha))$.

\smallskip
\noindent\emph{Step 1: gradient formula.}
For $\alpha>0$, dual optimality gives
\[
 \inv{p}=\frac{h_t(\inv{x})}{\alpha^2},
 \qquad
 \inv{q}=\frac{(g_t(\inv{x}))_+}{\alpha^2}.
\]
Define
\[
 \psi_\alpha^t(x)
 =f_t(x)+\frac{\norm{h_t(x)}^2}{2\alpha^2}
          +\frac{\norm{(g_t(x))_+}^2}{2\alpha^2}.
\]
Eliminating the dual variables shows that, along the selected local branch,
\[
 \varphi_{c,\alpha}^t(\ell)
 =\psi_\alpha^t(\inv{x})
  +\frac{1}{2c}\norm{\inv{x}-\ell}^2.
\]
Here $\inv{x}$ is the local minimizer near $\hat x(t)$ of
\[
 x\longmapsto \psi_\alpha^t(x)+\frac{1}{2c}\norm{x-\ell}^2,
\]
and therefore satisfies
\[
 \nabla\psi_\alpha^t(\inv{x})+\frac{\inv{x}-\ell}{c}=0.
\]
The chain-rule calculation in
\eqref{eq:regularized-envelope-chain-rule} and the subsequent extension
through the nondifferentiability points of the locally Lipschitz implicit
map give
\[
 \nabla\varphi_{c,\alpha}^t(\ell)
 =\frac{\ell-\inv{x}(t,\ell,\alpha)}c.
\]
For $\alpha=0$, this is Proposition~\ref{prop::gradient_computation};
Proposition~\ref{prop:implicit-solution-map} ensures compatibility of the
two formulas as $\alpha\downarrow0$.

\smallskip
\noindent\emph{Step 2: uniform local strong convexity.}
For $\alpha>0$, define the current positive-multiplier set by
\[
 \mathcal A=\mathcal A_+(t,\ell,\alpha)
 :=\bigl\{i\in\{1,\ldots,n_2\}:(\inv{q})^i>0\bigr\},
\]
and let $g_t^{\mathcal A}$ denote the subvector of $g_t$ indexed by
$\mathcal A$.  On any smooth region on which $\mathcal A$ is fixed, define
\[
\begin{aligned}
 H&=\nabla_{xx}^2\mathcal L^t(\inv{x},\inv{p},\inv{q}),\\
 M_{\mathcal A}
 &=\bigl[\nabla h_t(\inv{x})\ \ \nabla g_t^{\mathcal A}(\inv{x})\bigr],\\
 N_{\mathcal A}
 &=H+\frac{M_{\mathcal A}M_{\mathcal A}^T}{\alpha^2}.
\end{aligned}
\]
Differentiating the KKT system on that region gives the following
sensitivity formulas; see \eqref{current-active-sensitivity} for the
block-matrix derivation:
\[
 D_\ell\inv{x}=(I+cN_{\mathcal A})^{-1},
 \qquad
 D_\ell\nabla\varphi_{c,\alpha}^t
 =N_{\mathcal A}(I+cN_{\mathcal A})^{-1}.
\]
Apply the Debreu quadratic-form implication
\eqref{eq:debreu-implication} as follows.  Fix $t_0$ and set
$\mathcal S_0=\{i:\hat q^i(t_0)>0\}$.  Let $M_{\mathcal S_0}$ be
$M_{\mathcal A}$ with its inequality-gradient columns restricted to
$\mathcal S_0$, and evaluate
$(H_0,M_0)=(H,M_{\mathcal S_0})$ at
$(t,\ell,\alpha)=(t_0,\hat x(t_0),0)$.
Assumption~\ref{assumption::seven} says exactly that
$d^TH_0d>0$ for every $0\ne d\in\ker(M_0^T)$.  Debreu's implication
therefore provides $\rho_0,\gamma_0>0$ such that
\[
 H_0+\rho_0M_0M_0^T\succeq2\gamma_0I.
\]
Continuity then gives, on a sufficiently small parameter neighborhood,
\[
 H+\rho_0M_{\mathcal S_0}M_{\mathcal S_0}^T\succeq\gamma_0I,
 \qquad \mathcal S_0\subseteq\mathcal A.
\]
Therefore, whenever $\alpha\leq\rho_0^{-1/2}$,
\[
 N_{\mathcal A}
 =H+\frac{M_{\mathcal A}M_{\mathcal A}^T}{\alpha^2}
 \succeq H+\rho_0M_{\mathcal S_0}M_{\mathcal S_0}^T
 \succeq\gamma_0I.
\]
A finite cover of the compact reference curve, followed by taking the
smallest $\gamma_0$ and the smallest admissible upper bound on $\alpha$,
gives $N_{\mathcal A}\succeq\gamma I$ uniformly.  Recalling the
sensitivity formula above, on every smooth active-set region,
\[
 D_\ell\nabla\varphi_{c,\alpha}^t
 =N_{\mathcal A}(I+cN_{\mathcal A})^{-1}
 \succeq \frac{\gamma}{1+c\gamma}I=:\sigma I.
\]
The inequality follows from the monotonicity of
$\lambda\mapsto\lambda/(1+c\lambda)$; the spectral calculation is given
after \eqref{current-active-sensitivity}.  Thus the Debreu step turns the
second-order curvature on the constraint-tangent directions into a
uniform positive lower bound for the differential of the outer gradient.
At an active-set switch, the generalized Jacobian is formed from limits
and convex combinations of these matrices, so the same lower bound is
preserved. Continuity of the implicit map then extends the
strong-monotonicity inequality to $\alpha=0$.

\smallskip
\noindent\emph{Step 3: critical-point branch.}
Set
\[
 G(t,\ell,\alpha)
 :=\nabla\varphi_{c,\alpha}^t(\ell)
 =\frac{\ell-\inv{x}(t,\ell,\alpha)}c.
\]
By Step~2, $G(t,\cdot,\alpha)$ is uniformly strongly monotone, while
Proposition~\ref{prop:implicit-solution-map} supplies the required local
parameter regularity. Since $G(t,\hat x(t),0)=0$, the Implicit Function
Theorem for strongly monotone mappings gives a locally unique critical
point $\hat x_\alpha(t)$ depending locally Lipschitz-continuously on
$\alpha$. Finally, $G(t,\hat x_\alpha(t),\alpha)=0$ is equivalent to
\[
 \hat x_\alpha(t)
 =\inv{x}(t,\hat x_\alpha(t),\alpha).
\]
\end{proof}

\begin{remark}[Interpretation]
The proximal envelope converts a local constrained KKT branch
into a differentiable, locally convex outer objective. Its gradient is the
gap between the parameter $\ell$ and the implicit primal solution, divided
by $c$.
\end{remark}

\subsection{Two-time-scale flow and local tracking}
\label{sec:two-time-scale-flow}
We now combine the evolution of the primal--dual dynamics for the
Lagrangians $\mathcal{L}^{t,\ell}_{c,\alpha}$ with the evolution of $\ell$ given by a perturbed gradient descent. This way, while the
variables $(x,p,q)$ chase the solution of the $(t,\ell)$-problem,
$\ell$ tracks the solution of the regularized $t$-problem
$\hat{x}_{\alpha}(t)$, although these two dynamics will run at different
speed.  The convexity obtained in
Proposition~\ref{lemma::gradient_varphi} is local and implicit.  To impose
a prescribed uniform lower bound for the outer dynamics, we add the
regularization
\[
\varphi^t_{c,\alpha,\nu}(\ell)
=\varphi^t_{c,\alpha}(\ell)+\frac{\nu}{2}\norm{\ell}^2.
\]
For sufficiently small $\nu>0$, denote its locally selected critical
point by $\hat{x}_{\alpha,\nu}(t)$. \greenfix{The following remark records
the regularization bias needed below.}

\begingroup
\begin{remark}[Regularization bias]
\label{rem:nonconvex-regularization-bias}
The generalized Implicit Function Theorem used in
Proposition~\ref{prop:implicit-solution-map}, together with the local
strong convexity in Proposition~\ref{lemma::gradient_varphi}, makes the
selected outer critical point locally Lipschitz with respect to
perturbations of its defining equation.  Here
\[
 D(x,p,q,t,\ell,\alpha)-D(x,p,q,t,\ell,0)
 =\alpha^2(0,p,q),
\]
so the perturbation produced by dual regularization is $O(\alpha^2)$ on
the common local neighborhood.  The outer critical-point equation
\[
 \nabla\varphi_{c,\alpha}^t(\ell)+\nu\ell=0
\]
depends linearly on $\nu$.  Compactness of the reference curve therefore
gives uniform constants $C_\alpha,C_\nu>0$ and sufficiently small
parameter ranges such that
\[
 \norm{\hat x_\alpha(t)-\hat x(t)}\leq C_\alpha\alpha^2,
 \qquad
 \norm{\hat x_{\alpha,\nu}(t)-\hat x_\alpha(t)}\leq C_\nu\nu
\]
for every $t\in[0,t_f]$.  Consequently,
\begin{equation}
 \norm{\hat x_{\alpha,\nu}(t)-\hat x(t)}
 \leq C_\alpha\alpha^2+C_\nu\nu.
 \label{nonconvex-total-bias}
\end{equation}
After reducing the parameter ranges if necessary, both selected branches
remain inside the common local neighborhoods used below.
\end{remark}
\endgroup

For an admissible fixed $\alpha>0$, we abbreviate
$z^{\rm imp}(t,\ell):=z^{\rm imp}(t,\ell,\alpha)$ and set
\[
 z=(x,p,q),\qquad \hat\ell(t)=\hat x_{\alpha,\nu}(t),
\]
The local analysis directly tracks
\[
 \underbrace{\norm{\ell(t)-\hat\ell(t)}}_{\text{outer error}},\qquad
 \underbrace{\norm{z(t)-z^{\rm imp}(t,\ell(t))}}_{\text{inner error}}.
\]
We consider the coupled dynamics
\begin{align}
  \begin{split}\label{dynamics}
    \epsilon_1 \dot{x} &= -\displaystyle\frac{\partial
      \mathcal{L}^{t,\ell}_{c,\alpha}}{\partial x},
    \\
    \epsilon_1 \dot{p} &= \displaystyle\frac{\partial
      \mathcal{L}^{t,\ell}_{c,\alpha}}{\partial p},
    \\
    \epsilon_1 \dot{q} &= [\displaystyle\frac{\partial
      \mathcal{L}^{t,\ell}_{c,\alpha}}{\partial q}]^+_q,
    \\
    \epsilon_2\dot{\ell} &= \displaystyle\frac{x - \ell}{c} -\nu \ell.
  \end{split}
\end{align}
Notice that $\ell$ does not evolve following
$-\nabla\varphi^t_{c,\alpha,\nu}$, as we are using $x$ as opposed to
$\inv{x}$, and $x$ is evolving following the primal-dual
dynamics.
Figure~\ref{fig:two-time-scale-architecture} summarizes this feedback
architecture and the two errors used in the local analysis.

\begin{figure}[t]
\centering
\begin{tikzpicture}[
  font=\footnotesize,
  data/.style={draw=black!65, rounded corners=2pt, fill=black!3,
    align=center, text width=4.7cm, minimum height=0.85cm,
    inner sep=4pt},
  target/.style={draw=guidegreen, rounded corners=2pt, fill=guidegreen!7,
    align=center, text width=4.55cm, minimum height=0.9cm,
    inner sep=4pt},
  outerstate/.style={draw=guideblue, very thick, rounded corners=3pt,
    fill=guideblue!7, align=center, text width=4.75cm,
    minimum height=1.35cm, inner sep=5pt},
  innerstate/.style={draw=guideorange, very thick, rounded corners=3pt,
    fill=guideorange!8, align=center, text width=5.1cm,
    minimum height=1.35cm, inner sep=5pt},
  energy/.style={draw=black!55, rounded corners=2pt, fill=black!2,
    align=center, text width=10.6cm, minimum height=0.8cm,
    inner sep=5pt},
  outerpanel/.style={draw=guideblue!70, rounded corners=5pt,
    fill=guideblue!2, inner sep=9pt},
  innerpanel/.style={draw=guideorange!75, rounded corners=5pt,
    fill=guideorange!2, inner sep=9pt},
  tubebox/.style={draw=black!55, densely dashed, rounded corners=7pt,
    inner sep=12pt},
  dataarrow/.style={-{Latex[length=2mm]}, draw=black!65, thick},
  outerarrow/.style={-{Latex[length=2.2mm]}, draw=guideblue, very thick},
  innerarrow/.style={-{Latex[length=2.2mm]}, draw=guideorange, very thick},
  dependency/.style={-{Latex[length=2mm]}, draw=guidegreen!85!black,
    thick},
  feedback/.style={-{Latex[length=2.2mm]}, draw=guideblue!80!black,
    very thick}
]
\node[data] (problem) {\textbf{moving problem data}\\
  $t\mapsto(f_t,h_t,g_t)$};

\node[target, below=1.0cm of problem, xshift=-3.1cm] (outertarget)
  {\textbf{moving outer target}\\
   $\hat\ell(t)=\hat x_{\alpha,\nu}(t)$};
\node[target, below=1.0cm of problem, xshift=3.1cm] (innertarget)
  {\textbf{instantaneous inner target}\\
   $z^{\rm imp}(t,\ell)$};

\node[outerstate, below=1.0cm of outertarget] (outerstate)
  {\textbf{slow outer state $\ell$}\\[-1mm]
   $\epsilon_2\dot\ell=(x-\ell)/c-\nu\ell$};
\node[innerstate, below=1.0cm of innertarget] (innerstate)
  {\textbf{fast primal--dual state $z=(x,p,q)$}\\[-1mm]
   $\Lc=\mathcal L_{c,\alpha}^{t,\ell}$\\[-1mm]
   $\epsilon_1\dot z=(-\nabla_x\Lc,\nabla_p\Lc,
      [\nabla_q\Lc]^+_q)$};

\coordinate (statemid) at ($(outerstate.south)!0.5!(innerstate.south)$);
\node[energy, below=1.3cm of statemid] (energy)
  {$\displaystyle
    \underbrace{\|\ell-\hat\ell(t)\|}_{\text{outer error}},
    \qquad
    \underbrace{\|z-z^{\rm imp}(t,\ell)\|}_{\text{inner error}}$};

\draw[dataarrow] (problem.south west) to[out=-125,in=90]
  (outertarget.north);
\draw[dataarrow] (problem.south east) to[out=-55,in=90]
  (innertarget.north);
\draw[outerarrow] (outertarget) --
  node[left, align=right, text=guideblue] {outer tracking\\scale $\epsilon_2$}
  (outerstate);
\draw[innerarrow] (innertarget) --
  node[right, align=left, text=guideorange] {inner tracking\\scale $\epsilon_1$}
  (innerstate);
\draw[dependency] (outerstate.east) to[out=20,in=200]
  node[above, sloped, text=guidegreen!55!black]
  {$\ell$ sets the proximal center} (innertarget.west);
\draw[feedback] (innerstate.south) -- ++(0,-0.55) -|
  node[pos=0.25, below, text=guideblue!80!black]
  {primal feedback $x$} (outerstate.south);

\begin{scope}[on background layer]
  \node[outerpanel, fit=(outertarget)(outerstate)] (outerpanel) {};
  \node[innerpanel, fit=(innertarget)(innerstate)] (innerpanel) {};
  \node[tubebox, fit=(problem)(outerpanel)(innerpanel)(energy)]
    (trackingtube) {};
\end{scope}
\node[font=\bfseries\footnotesize, text=guideblue, anchor=south west]
  at ([xshift=2mm,yshift=1mm]outerpanel.north west) {Outer loop};
\node[font=\bfseries\footnotesize, text=guideorange, anchor=south east]
  at ([xshift=-2mm,yshift=1mm]innerpanel.north east) {Inner loop};
\node[font=\scriptsize, text=black!65, anchor=south east]
  at ([xshift=-2mm,yshift=1mm]trackingtube.south east)
  {local tube for the tracking theorem};
\end{tikzpicture}
\caption{Two-time-scale feedback architecture in
Section~\ref{sec:two-time-scale-flow}.  For the current pair $(t,\ell)$,
the fast projected primal--dual dynamics drives $z$ toward the implicit
saddle $z^{\rm imp}(t,\ell)$.  Its primal component $x$ closes the slow
outer loop for $\ell$, whose moving target is
$\hat x_{\alpha,\nu}(t)$.  A common admissible radius for the two
displayed errors is constructed below.}
\label{fig:two-time-scale-architecture}
\end{figure}
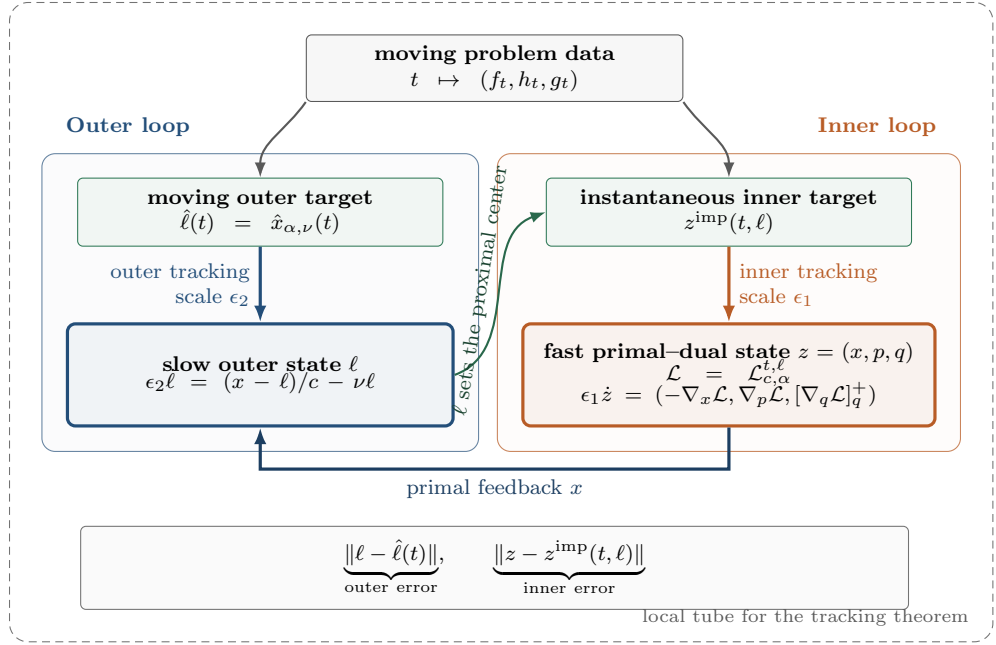

We next construct uniform neighborhoods in which the two tracking errors
can be controlled. They are chosen to guarantee the following properties.
\begin{itemize}
\item The outer state $\ell$ tracks $\hat x_{\alpha,\nu}(t)$ and remains
inside the neighborhood where the outer envelope is strongly convex.
\item The inner state $z=(x,p,q)$ tracks $z^{\rm imp}(t,\ell)$ and
remains near the reference KKT curve, where the local Lagrangian is
strongly convex--strongly concave.
\end{itemize}
Let
$C_{\min}:=\{(\hat x(t),\hat p(t),\hat q(t)):t\in[0,t_f]\}$
be the reference KKT curve on $[0,t_f]$, and suppose that
Assumptions~\ref{assumption::seven} and
\ref{assumption::time-regularity} hold along it. Choose $c>0$ small enough that
$\mathcal{L}_{c,\alpha}^{t,\ell}$ is $\mu$-strongly convex with respect
to $x$, for some uniform $\mu>0$, on a neighborhood $N_1$ of
$C_{\min}$ and for the nearby values of $\ell$ used below. Without loss
of generality, take $r_1>0$ such that $N_1$ has the form
\[
N_1 = \bigcup_{t\in [0,t_f]} B((\hat x(t), \hat p(t), \hat q(t));r_1).
\]
By compactness and the uniform local Lipschitz continuity of the implicit
mapping, there exist
$r_2>0$ and $\overline{\alpha}>0$ such that
\[
N_2= \bigcup_{t\in [0 ,t_f ]}\{t\}\times B(\hat x(t);r_2) \times(-\overline\alpha,\overline\alpha)
\]
is mapped by $(\inv{x},\inv{p},\inv{q})$ into
$\cup_{t\in[0,t_f]}B((\hat x(t),\hat p(t),\hat q(t));r_1/2)$.
Reduce $\overline\alpha$ if necessary so that it also lies within the
uniform $\alpha$-ranges permitted by
Proposition~\ref{lemma::gradient_varphi} and
Remark~\ref{rem:nonconvex-regularization-bias}.
On this tube the implicit map obeys one joint estimate
\begin{equation}\label{uniform-implicit-lipschitz}
\norm{z^{\rm imp}(t_2,\ell_2,\alpha')
      -z^{\rm imp}(t_1,\ell_1,\alpha')}
\leq e_1|t_2-t_1|+e_2\norm{\ell_2-\ell_1}
\end{equation}
for finite constants $e_1,e_2$, uniformly for
$|\alpha'|<\overline\alpha$, and all pairs whose connecting parameter segment
remains in the tube.

 Once more, by compactness and
Proposition~\ref{lemma::gradient_varphi}, choose
$0<\alpha<\overline\alpha$ and $r_3>0$ such that
$\varphi_{c,\alpha}^t$ is defined and strongly convex in
$B(\hat x_{\alpha}(t);r_3)
\subseteq B(\hat x(t);r_2)$ for $t\in[0,t_f]$.
Using compactness of $[0,t_f]$ and
Remark~\ref{rem:nonconvex-regularization-bias}, choose an admissible
$\nu>0$ and $r_4>0$
such that
$B(\hat x_{\alpha,\nu}(t);r_4)
\subseteq B(\hat x_{\alpha}(t);r_3)$ for all $t\in[0,t_f]$.
Define the explicit tube margin
\begin{equation}\label{nonconvex-tube-radius}
 R=\min\{r_1/2,r_4\}.
\end{equation}
\begingroup
The role of $R$ can be stated directly in terms of the two geometric
errors. If $\norm{\ell(t)-\hat\ell(t)}<R\leq r_4$,
then $\ell$ remains in the outer strong-convexity neighborhood. Moreover,
$z^{\rm imp}(t,\ell)$ lies within $r_1/2$ of the reference KKT point
$\hat z(t)=(\hat x(t),\hat p(t),\hat q(t))$. Hence, if
$\norm{z(t)-z^{\rm imp}(t,\ell(t))}<R\leq r_1/2$,
the triangle inequality gives
$\norm{z-\hat z(t)}<r_1$, and therefore $z\in N_1$.
\endgroup
\FloatBarrier
Figure~\ref{illustration} illustrates the construction at a fixed time.
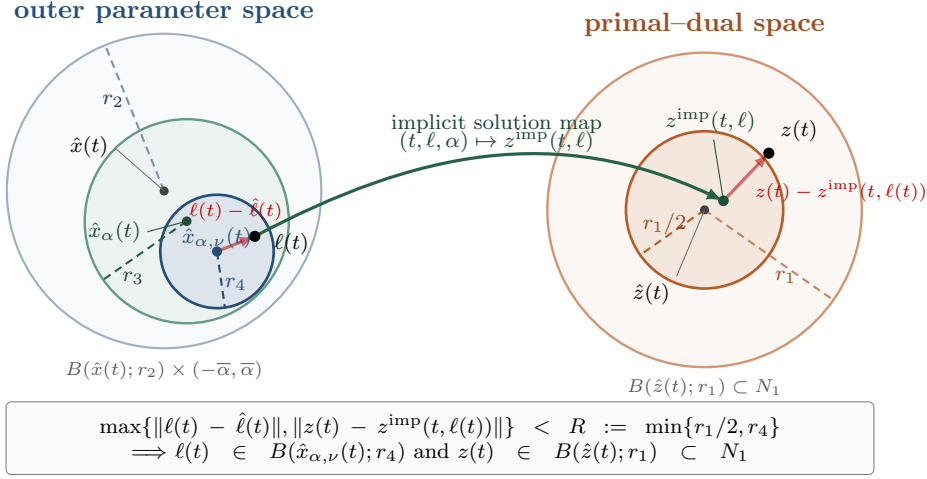
\begin{figure}[t]
\centering
\begin{tikzpicture}[
  font=\footnotesize,
  maparrow/.style={-{Latex[length=2.2mm]}, very thick,
    draw=guidegreen!80!black},
  errorarrow/.style={-{Latex[length=2mm]}, very thick,
    draw=red!75!black, draw opacity=0.55},
  errorlabel/.style={font=\scriptsize, text=red!75!black,
    text opacity=1, inner sep=1pt},
  radiusline/.style={densely dashed, thick},
  leader/.style={draw=black!55, thin},
  pointlabel/.style={text opacity=1, inner sep=1.4pt},
  mathlabel/.style={text opacity=1, inner sep=1pt},
  callout/.style={draw=black!50, rounded corners=2pt, fill=black!2,
    align=center, text width=11.2cm, inner sep=4pt}
]
\coordinate (xref) at (-3.65,0.25);
\coordinate (xalpha) at (-3.35,-0.15);
\coordinate (ellref) at (-2.95,-0.55);
\coordinate (ellstate) at (-2.45,-0.35);

\coordinate (zref) at (3.5,0);
\coordinate (zimp) at (3.75,0.12);
\coordinate (zstate) at (4.35,0.75);

\filldraw[fill=guideblue!3, draw=guideblue!45, line width=0.8pt]
  (xref) circle (2.08cm);
\filldraw[fill=guidegreen!9, draw=guidegreen!75, line width=0.9pt]
  (xalpha) circle (1.35cm);
\filldraw[fill=guideblue!15, draw=guideblue, line width=1pt]
  (ellref) circle (0.75cm);

\filldraw[fill=guideorange!5, draw=guideorange!65, line width=0.9pt]
  (zref) circle (2.08cm);
\filldraw[fill=guideorange!15, draw=guideorange, line width=1pt]
  (zref) circle (1.04cm);

\draw[maparrow] (ellstate) to[out=28,in=155]
  node[mathlabel, pos=0.52, above, align=center,
    text=guidegreen!55!black]
  {implicit solution map\\[-1mm]
   $(t,\ell,\alpha)\mapsto z^{\rm imp}(t,\ell)$}
  (zimp);

\draw[radiusline, draw=guideblue!60] (xref) -- ++(112:2.08cm)
  node[mathlabel, pos=0.62, left, text=guideblue!75!black] {$r_2$};
\draw[radiusline, draw=guidegreen!80!black] (xalpha) -- ++(-145:1.35cm)
  node[mathlabel, pos=0.64, below, yshift=-1.2mm,
    text=guidegreen!55!black] {$r_3$};
\draw[radiusline, draw=guideblue] (ellref) -- ++(-82:0.75cm)
  node[mathlabel, pos=0.58, right, text=guideblue] {$r_4$};

\draw[radiusline, draw=guideorange!75] (zref) -- ++(-35:2.08cm)
  node[mathlabel, pos=0.63, below, text=guideorange!80!black] {$r_1$};
\draw[radiusline, draw=guideorange] (zref) -- ++(215:1.04cm)
  node[mathlabel, pos=0.62, above, text=guideorange!80!black] {$r_1/2$};

\draw[errorarrow] (ellref) -- (ellstate)
  node[errorlabel, pos=0.5, above, yshift=2.5mm]
  {$\ell(t)-\hat\ell(t)$};
\draw[errorarrow] (zimp) -- (zstate)
  node[errorlabel, pos=0.52, below right, xshift=0.5mm]
  {$z(t)-z^{\rm imp}(t,\ell(t))$};

\fill[black!75] (xref) circle (1.7pt);
\draw[leader] (xref) -- (-4.30,0.82);
\node[pointlabel, anchor=east] at (-4.33,0.84) {$\hat x(t)$};
\fill[guidegreen!65!black] (xalpha) circle (1.8pt);
\draw[leader, draw=guidegreen!65!black] (xalpha) -- (-3.87,-0.29);
\node[pointlabel, anchor=east, text=guidegreen!55!black]
  at (-3.90,-0.29) {$\hat x_{\alpha}(t)$};
\fill[guideblue] (ellref) circle (1.9pt);
\node[pointlabel, anchor=south, text=guideblue]
  at ([yshift=-0.2mm]ellref) {$\hat x_{\alpha,\nu}(t)$};
\fill[black] (ellstate) circle (2.1pt);
\node[pointlabel, anchor=west]
  at ([xshift=2mm,yshift=-1mm]ellstate) {$\ell(t)$};

\fill[black!75] (zref) circle (1.7pt);
\draw[leader] (zref) -- (3.13,-0.88);
\node[pointlabel, anchor=north east] at (3.10,-0.91) {$\hat z(t)$};
\fill[guidegreen!65!black] (zimp) circle (2pt);
\draw[leader, draw=guidegreen!65!black] (zimp) -- (3.62,0.93);
\node[pointlabel, anchor=south, text=guidegreen!55!black]
  at (3.58,0.97) {$z^{\rm imp}(t,\ell)$};
\fill[black] (zstate) circle (2.1pt);
\node[pointlabel, anchor=south west]
  at ([xshift=1mm,yshift=1mm]zstate) {$z(t)$};

\node[font=\bfseries, text=guideblue] at (-3.65,2.62)
  {outer parameter space};
\node[font=\bfseries, text=guideorange!85!black] at (3.5,2.45)
  {primal--dual space};
\node[mathlabel, font=\scriptsize, align=center, text=black!65]
  at (-3.65,-2.12)
  {$B(\hat x(t);r_2)\times
    (-\overline\alpha,\overline\alpha)$};
\node[mathlabel, font=\scriptsize, align=center, text=black!65]
  at (3.5,-2.35)
  {$B(\hat z(t);r_1)\subset N_1$};
\node[callout] at (0,-3.02)
  {$\max\{\|\ell(t)-\hat\ell(t)\|,
             \|z(t)-z^{\rm imp}(t,\ell(t))\|\}<R:=\min\{r_1/2,r_4\}$
   $\Longrightarrow$
   $\ell(t)\in B(\hat x_{\alpha,\nu}(t);r_4)$ and
   $z(t)\in B(\hat z(t);r_1)\subset N_1$};
\end{tikzpicture}
\caption{Nested neighborhoods used in the local analysis at a fixed time
$t$.  The selected outer target
$\hat x_{\alpha,\nu}(t)$ lies in the nested
$r_4$-, $r_3$-, and $r_2$-neighborhoods.  The implicit map sends
$(t,\ell,\alpha)$ to $z^{\rm imp}(t,\ell)$ inside the
$r_1/2$-ball about $\hat z(t)$.  The red arrows show the outer error
$\ell(t)-\hat\ell(t)$ and the inner error
$z(t)-z^{\rm imp}(t,\ell(t))$.  The common margin $R$ keeps both the outer
and primal--dual states inside the neighborhoods where the local analysis
applies.}
\label{illustration}
\end{figure}

\begin{lemma}\label{lemma::ten} Under
Assumption~\ref{assumption::time-regularity}, the curve
$\hat x_{\alpha,\nu}(t)$ is Lipschitz in $t$. Consequently,
it is differentiable almost everywhere and its derivative is essentially
bounded on $[0,t_f]$.
\end{lemma}
\begin{proof}
Write
\[
F_t(\ell)=\nabla\varphi_{c,\alpha,\nu}^t(\ell)
=\frac{\ell-\inv{x}(t,\ell,\alpha)}c
+\nu\ell.
\]
On the uniform tube, $F_t$ is at least $\nu$-strongly monotone in
$\ell$, and \eqref{uniform-implicit-lipschitz} gives a constant $L_T$
such that
\[
\norm{F_{t_2}(\ell)-F_{t_1}(\ell)}\leq L_T|t_2-t_1|.
\]
Let $\ell_i=\hat x_{\alpha,\nu}(t_i)$, so
$F_{t_1}(\ell_1)=0$ and $F_{t_2}(\ell_2)=0$.  Notice that
$F_{t_1}(\ell_2)$ need not vanish because the mapping itself depends on
time.  Therefore,
\[
F_{t_1}(\ell_2)-F_{t_1}(\ell_1)
=F_{t_1}(\ell_2)-F_{t_2}(\ell_2).
\]
Strong monotonicity and Cauchy--Schwarz now give
\begin{align*}
\nu\norm{\ell_2-\ell_1}^2
&\leq\langle F_{t_1}(\ell_2)-F_{t_1}(\ell_1),
                 \ell_2-\ell_1\rangle\\
&=\langle F_{t_1}(\ell_2)-F_{t_2}(\ell_2),
                 \ell_2-\ell_1\rangle\\
&\leq L_T|t_2-t_1|\norm{\ell_2-\ell_1}.
\end{align*}
Thus $\norm{\ell_2-\ell_1}\leq(L_T/\nu)|t_2-t_1|$ for nearby
times.  A finite partition of $[t_1,t_2]$, using the uniform tube and
compactness, gives the same estimate for arbitrary $t_1,t_2\in[0,t_f]$.
Rademacher's theorem gives differentiability almost everywhere and the
same constant is an essential bound for the derivative.
\end{proof}

\begin{theorem}[Local tracking in the nonconvex case]
\label{proposition::lyapunov}
Let $[0,t_f]$ be an interval on which
Assumptions~\ref{assumption::seven} and
\ref{assumption::time-regularity} hold, and carry out the neighborhood
construction above with the selected
\greenfix{$c>0$, $\alpha>0$, and $\nu>0$}. Write
\[
z=(x,p,q),\qquad
z^{\rm imp}(t,\ell)
=\bigl(\inv{x}(t,\ell,\alpha),\inv{p}(t,\ell,\alpha),
        \inv{q}(t,\ell,\alpha)\bigr),
\]
and let $\hat\ell(t)=\hat x_{\alpha,\nu}(t)$. On the resulting uniform
tube, let $e_1,e_2$ be any finite uniform constants satisfying
\eqref{uniform-implicit-lipschitz}, and make the conservative choices
\begin{equation}\label{explicit-nonconvex-e-constants}
 e_3:=\frac{e_1}{c\nu},
 \qquad
 e_4:=\frac{1+e_2}{c}+\nu.
\end{equation}
As verified below, these constants satisfy
\begin{align*}
\norm{\dot{\hat\ell}(t)}&\leq e_3
 &&\text{for almost every }t,\\
\norm{\nabla\varphi_{c,\alpha,\nu}^t(\ell)}
&\leq e_4\norm{\ell-\hat\ell(t)}
 &&\text{throughout the tube}.
\end{align*}
Let $m=\min\{\mu,\alpha^2\}$. For arbitrary $\beta_1,\beta_2>0$,
define
\begin{align*}
a_1&=-\frac{\nu}{\epsilon_2}+\frac{\beta_1}{2}
       +\frac{\beta_2e_4^2}{2},\\
a_2&=-\frac{m}{\epsilon_1}
 +\frac{1}{2\beta_1c^2\epsilon_2^2}
 +\frac{e_2^2}{2\beta_2\epsilon_2^2}
 +\frac{e_2}{c\epsilon_2}.
\end{align*}
Choose $\epsilon_2>0$ sufficiently small that $a_1<0$, and then choose
$\epsilon_1>0$ sufficiently small that $a_2<0$. Define
\begin{equation}\label{nonconvex-AB-constants}
A:=-\max\{a_1,a_2\}>0,
\qquad
B:=\sqrt{e_1^2+e_3^2}.
\end{equation}

For $q(0)\geq0$, define
\[
V(t)=\norm{\ell(t)-\hat\ell(t)}^2
+\norm{z(t)-z^{\rm imp}(t,\ell(t))}^2.
\]
Assume $\sqrt{V(0)}<R$, where $R$ is defined in
\eqref{nonconvex-tube-radius}.  As long as the trajectory remains in the
constructed tube,
\begin{equation}\label{nonconvex-dissipation-corrected}
\frac12\dot V(t)\leq-AV(t)+B\sqrt{V(t)}
\quad\text{for almost every }t,
\end{equation}
and hence
\begin{equation}\label{nonconvex-tracking-bound-corrected}
\sqrt{V(t)}\leq e^{-At}\sqrt{V(0)}
+\frac BA(1-e^{-At}).
\end{equation}
Since $\norm{\ell(t)-\hat x_{\alpha,\nu}(t)}\leq\sqrt{V(t)}$, the
outer state obeys the explicit tracking bound
\begin{align}
\norm{\ell(t)-\hat x_{\alpha,\nu}(t)}
&\leq e^{-At}
\Bigl(\norm{\ell(0)-\hat x_{\alpha,\nu}(0)}^2 \notag\\
&\qquad\quad
+\norm{z(0)-z^{\rm imp}(0,\ell(0))}^2\Bigr)^{1/2}
+\frac BA(1-e^{-At}).
\label{nonconvex-outer-tracking-bound}
\end{align}
In particular, the explicit condition
\begin{equation}\label{nonconvex-invariance-condition}
 \max\left\{\sqrt{V(0)},\frac BA\right\}<R
\end{equation}
ensures that the trajectory remains in the tube, so the estimate is valid
on all of $[0,t_f]$.
\end{theorem}

\begin{proof}
The projected vector field is locally Lipschitz on the tube in the
projected-dynamical-system sense, so the same local Carath\'eodory
existence and uniqueness argument as in Theorem~\ref{th:tracking} applies.
We first justify the concrete constants in
\eqref{explicit-nonconvex-e-constants}.  At a fixed $\ell$ and for
admissible nearby times, the temporal part of
\eqref{uniform-implicit-lipschitz} gives
\[
\norm{\inv{x}(t_2,\ell,\alpha)-\inv{x}(t_1,\ell,\alpha)}
\leq e_1|t_2-t_1|.
\]
Hence the constant $L_T$ in Lemma~\ref{lemma::ten} can be chosen as
$e_1/c$, and that lemma yields
$\norm{\dot{\hat\ell}(t)}\leq e_1/(c\nu)=e_3$ almost everywhere.
Moreover, the segment joining $\ell$ and $\hat\ell$ remains in the
localized ball. Since
$\nabla\varphi_{c,\alpha,\nu}^t(\hat\ell)=0$, the spatial part of
\eqref{uniform-implicit-lipschitz} therefore gives
\begin{align}
\norm{\nabla\varphi_{c,\alpha,\nu}^t(\ell)}
&=\left\|\frac{(\ell-\hat\ell)
 -\bigl(\inv{x}(t,\ell,\alpha)
          -\inv{x}(t,\hat\ell,\alpha)\bigr)}c
 +\nu(\ell-\hat\ell)\right\|\notag\\
&\leq\left(\frac{1+e_2}{c}+\nu\right)
 \norm{\ell-\hat\ell}
=e_4\norm{\ell-\hat\ell}.
\label{nonconvex-envelope-gradient-bound}
\end{align}
\begingroup
For compactness in the following calculation, write
$x^{\rm imp}:=\inv{x}(t,\ell,\alpha)$.
Since $\hat\ell(t)$ minimizes $\varphi_{c,\alpha,\nu}^t$,
$\nabla\varphi_{c,\alpha,\nu}^t(\hat\ell)=0$.  Strong convexity, with
modulus at least $\nu$, therefore gives
\[
\left\langle \ell-\hat\ell,
 \nabla\varphi_{c,\alpha,\nu}^t(\ell)
 -\nabla\varphi_{c,\alpha,\nu}^t(\hat\ell)\right\rangle
\geq \nu\norm{\ell-\hat\ell}^2.
\]
Recall that the slow dynamics gives
$\dot\ell=(x-\ell)/(c\epsilon_2)-(\nu/\epsilon_2)\ell$.
Moreover, almost everywhere,
$\dot{\hat\ell}(t)=\frac{d}{dt}\hat x_{\alpha,\nu}(t)$ and
$\norm{\dot{\hat\ell}(t)}\leq e_3$.
Consequently, using
$\norm{x-x^{\rm imp}}
\leq\norm{z-z^{\rm imp}(t,\ell)}$, we obtain
\begin{alignat}{3}
\textup{(i)}\quad
\frac12\frac d{dt}\norm{\ell-\hat\ell}^2
&=\langle \ell-\hat\ell,\dot\ell\rangle
&\,&{}
&\,&{}-\langle \ell-\hat\ell,\dot{\hat\ell}\rangle\notag\\
\textup{(ii)}\quad
&=\frac1{c\epsilon_2}\langle \ell-\hat\ell,x-\ell\rangle
&\,&{}-\frac\nu{\epsilon_2}\langle \ell-\hat\ell,\ell\rangle
&\,&{}-\langle \ell-\hat\ell,\dot{\hat\ell}\rangle\notag\\
\textup{(iii)}\quad
&=\frac1{c\epsilon_2}\langle \ell-\hat\ell,x-x^{\rm imp}\rangle
&\,&{}-\frac1{\epsilon_2}
 \left\langle \ell-\hat\ell,\frac{\ell-x^{\rm imp}}c+\nu\ell\right\rangle
&\,&{}-\langle \ell-\hat\ell,\dot{\hat\ell}\rangle\notag\\
\textup{(iv)}\quad
&=\frac1{c\epsilon_2}\langle \ell-\hat\ell,x-x^{\rm imp}\rangle
&\,&{}-\frac1{\epsilon_2}
 \left\langle \ell-\hat\ell,\nabla\varphi_{c,\alpha,\nu}^t(\ell)\right\rangle
&\,&{}-\langle \ell-\hat\ell,\dot{\hat\ell}\rangle\notag\\
\textup{(v)}\quad
&\leq\frac{\norm{\ell-\hat\ell}\norm{z-z^{\rm imp}(t,\ell)}}
 {c\epsilon_2}
&\,&{}-\frac\nu{\epsilon_2}\norm{\ell-\hat\ell}^2
&\,&{}+e_3\norm{\ell-\hat\ell}\notag\\
\textup{(vi)}\quad
&\leq\frac{\norm{z-z^{\rm imp}(t,\ell)}^2}
 {2\beta_1c^2\epsilon_2^2}
&\,&{}+\left(-\frac\nu{\epsilon_2}+\frac{\beta_1}{2}\right)
 \norm{\ell-\hat\ell}^2
&\,&{}+e_3\norm{\ell-\hat\ell}.
\label{outer-error-derivative}
\end{alignat}
Steps~\textup{(i)}--\textup{(v)} use, respectively, the chain rule, the
slow dynamics, addition and subtraction of $x^{\rm imp}$, the
envelope-gradient identity, and strong convexity with Cauchy--Schwarz;
their three contributions remain column-aligned. In \textup{(vi)},
Young's inequality splits the first-column product, and its
$(\beta_1/2)\norm{\ell-\hat\ell}^2$ contribution combines with the
second-column contraction term.
\par
\endgroup
The composition $t\mapsto z^{\rm imp}(t,\ell(t))$ is absolutely
continuous.  The metric estimate
\eqref{uniform-implicit-lipschitz}, rather than a classical chain rule on
the possible nondifferentiability set, gives
\[
\left|\frac d{dt}z^{\rm imp}(t,\ell(t))\right|
\leq e_1+e_2\norm{\dot\ell(t)}
\quad\text{for almost every }t.
\]
The inner Lagrangian is $\mu$-strongly convex in $x$ and
$\alpha^2$-strongly concave in $(p,q)$. Applying the same
add-and-subtract saddle-gap argument as in the proof of
Theorem~\ref{th:tracking}, together with the metric-derivative bound,
yields
\begin{align*}
\frac12\frac d{dt}\norm{z-z^{\rm imp}(t,\ell)}^2
&\leq-\frac m{\epsilon_1}\norm{z-z^{\rm imp}(t,\ell)}^2
 +e_1\norm{z-z^{\rm imp}(t,\ell)}\\*
&\quad+e_2\norm{z-z^{\rm imp}(t,\ell)}\norm{\dot\ell}.
\end{align*}
Furthermore, adding and subtracting
$\inv{x}(t,\ell,\alpha)$ in the slow update gives, for almost every $t$,
\begin{align*}
\dot\ell
&=\frac{x-\ell}{c\epsilon_2}-\frac{\nu}{\epsilon_2}\ell\\
&=\frac{x-\inv{x}(t,\ell,\alpha)}{c\epsilon_2}
 -\frac1{\epsilon_2}
 \left(\frac{\ell-\inv{x}(t,\ell,\alpha)}c+\nu\ell\right)\\
&=\frac{x-\inv{x}(t,\ell,\alpha)}{c\epsilon_2}
 -\frac1{\epsilon_2}\nabla\varphi_{c,\alpha,\nu}^t(\ell).
\end{align*}
The triangle inequality, the gradient bound
\eqref{nonconvex-envelope-gradient-bound}, and the fact that the
primal-component error is bounded by the full primal--dual error now give
\begin{align*}
\norm{\dot\ell}
&\leq\frac1{\epsilon_2}
 \norm{\nabla\varphi_{c,\alpha,\nu}^t(\ell)}
 +\frac1{c\epsilon_2}
 \norm{x-\inv{x}(t,\ell,\alpha)}\\
&\leq\frac{e_4}{\epsilon_2}\norm{\ell-\hat\ell}
 +\frac1{c\epsilon_2}\norm{z-z^{\rm imp}(t,\ell)}.
\end{align*}
Applying Young's inequality to the resulting product of the two error
norms gives
\begin{align}\label{inner-error-derivative}
\frac12\frac d{dt}\norm{z-z^{\rm imp}(t,\ell)}^2
&\leq\frac{\beta_2e_4^2}{2}\norm{\ell-\hat\ell}^2 \notag\\
&\quad+\left(-\frac m{\epsilon_1}
+\frac{e_2^2}{2\beta_2\epsilon_2^2}
+\frac{e_2}{c\epsilon_2}\right)
 \norm{z-z^{\rm imp}(t,\ell)}^2 \notag\\
&\quad+e_1\norm{z-z^{\rm imp}(t,\ell)}.
\end{align}
Adding \eqref{outer-error-derivative} and
\eqref{inner-error-derivative}, and using
\begin{align*}
e_3\norm{\ell-\hat\ell}
+e_1\norm{z-z^{\rm imp}(t,\ell)}
&\leq B\sqrt{\norm{\ell-\hat\ell}^2
 +\norm{z-z^{\rm imp}(t,\ell)}^2}\\
&=B\sqrt V,
\end{align*}
proves
\eqref{nonconvex-dissipation-corrected}. The comparison argument for
$\sqrt V$ proves \eqref{nonconvex-tracking-bound-corrected}.  The
outer-state estimate \eqref{nonconvex-outer-tracking-bound} follows from
$\norm{\ell-\hat\ell}\leq\sqrt V$ and the definition of $V(0)$. The
asserted choice of time scales is possible because $a_1\to-\infty$
as $\epsilon_2\downarrow0$, after which $a_2\to-\infty$ as
$\epsilon_1\downarrow0$.  Finally, under
\eqref{nonconvex-invariance-condition}, the right-hand side of
\eqref{nonconvex-tracking-bound-corrected} is a convex combination of two
numbers strictly smaller than $R$.  A first-exit argument therefore rules
out departure from the tube and extends the local solution and estimate
through $[0,t_f]$.
\end{proof}

\begingroup
\begin{corollary}[Actual primal tracking error]
\label{cor:nonconvex-actual-primal-error}
Under the hypotheses of Theorem~\ref{proposition::lyapunov}, let
\[
 M_\ell:=\sup_{t\in[0,t_f]}\norm{\hat x_{\alpha,\nu}(t)}<\infty.
\]
Then the primal component of the two-time-scale flow satisfies
\begin{align}
 \norm{x(t)-\hat x(t)}
 &\leq (1+e_2)\left(
 e^{-At}\sqrt{V(0)}+\frac BA(1-e^{-At})\right) \notag\\
 &\qquad{}+C_\alpha\alpha^2+(C_\nu+cM_\ell)\nu,
 \label{nonconvex-actual-primal-bound}
\end{align}
where $C_\alpha$ and $C_\nu$ are defined in
Remark~\ref{rem:nonconvex-regularization-bias}. Thus the actual primal
error is the sum of a decaying transient, a moving-target tracking floor,
and the regularization bias.
\end{corollary}
\begin{proof}
Write $x^{\rm imp}(t,\ell)$ for the primal component of
$z^{\rm imp}(t,\ell)$. The triangle inequality and
\eqref{uniform-implicit-lipschitz} give
\begin{align*}
 \norm{x-\hat x}
 &\leq \norm{x-x^{\rm imp}(t,\ell)}
 +\norm{x^{\rm imp}(t,\ell)-x^{\rm imp}(t,\hat\ell)}\\
 &\qquad{}+\norm{x^{\rm imp}(t,\hat\ell)-\hat\ell}
 +\norm{\hat\ell-\hat x}\\
 &\leq (1+e_2)\sqrt V
 +c\nu\norm{\hat\ell}
 +C_\alpha\alpha^2+C_\nu\nu.
\end{align*}
Indeed, the outer critical-point equation implies
$x^{\rm imp}(t,\hat\ell)=(1+c\nu)\hat\ell$.
The claim now follows from
\eqref{nonconvex-tracking-bound-corrected} and the definition of $M_\ell$.
\end{proof}
\endgroup

\begin{remark}[Interpretation]
The two terms in $V$ separate outer error
$\ell-\hat x_{\alpha,\nu}$ from inner error
$z-z^{\rm imp}(t,\ell)$. The rate $A$ is set by curvature and the chosen
time scales; $B$ measures motion of the target and implicit solution map.
Unlike Theorem~\ref{th:tracking}, this guarantee is local; condition
\eqref{nonconvex-invariance-condition} is an explicit sufficient test for
remaining inside the constructed tube.
\end{remark}

\begin{remark}[Time-scale separation versus target motion]
The ratio
\[
\rho=\frac{\epsilon_1}{\epsilon_2}
\]
controls the separation between the fast primal--dual state $z$ and the
slow outer state $\ell$, but it is not the only relevant quantity for a
time-dependent problem. Indeed, under the rescaled time
$\tau=t/\epsilon_2$, the dynamics of a frozen problem depend on
$(\epsilon_1,\epsilon_2)$ only through $\rho$; a common rescaling of both
parameters merely changes the clock. For moving data, however, the problem
is evaluated at $t=\epsilon_2\tau$, so the absolute value of $\epsilon_2$
also determines the speed of the algorithm relative to the moving target.

\end{remark}

\vspace*{0.25cm}

The next result records the effect of bounded additive disturbances.
\begin{proposition}[Robustness in the nonconvex case]
\label{prop:nonconvex-robustness}
Replace \eqref{dynamics} by
\begin{align}
  \begin{split}\label{dynamics_error}
    \epsilon_1 \dot{x} &= -\displaystyle\frac{\partial
      \mathcal{L}^{t,\ell}_{c,\alpha}}{\partial x} + e_x(t),
    \\
    \epsilon_1 \dot{p} &= \displaystyle\frac{\partial
      \mathcal{L}^{t,\ell}_{c,\alpha}}{\partial p}+ e_p(t),
    \\
    \epsilon_1 \dot{q} &= [\displaystyle\frac{\partial
      \mathcal{L}^{t,\ell}_{c,\alpha}}{\partial q} + e_q(t)]^+_q,
    \\
    \epsilon_2\dot{\ell} &= \displaystyle\frac{x - \ell}{c} -\nu \ell+ e_\ell(t),
  \end{split}
\end{align}
where the errors are measurable and, almost everywhere,
\[
\norm{(e_x,e_p,e_q)}\leq\delta_1,
\qquad \norm{e_\ell}\leq\delta_2.
\]
Under the hypotheses and notation of Theorem~\ref{proposition::lyapunov},
let
\[
E=\sqrt{
\left(\frac{\delta_2}{\epsilon_2}\right)^2+
\left(\frac{\delta_1}{\epsilon_1}
+\frac{e_2\delta_2}{\epsilon_2}\right)^2}.
\]
As long as the trajectory remains in the constructed tube,
\[
\frac12\dot V\leq-AV+(B+E)\sqrt V,
\]
and therefore
\[
\sqrt{V(t)}\leq e^{-At}\sqrt{V(0)}
+\frac{B+E}{A}(1-e^{-At}).
\]
If
\begin{equation}\label{nonconvex-robust-invariance-condition}
 \max\left\{\sqrt{V(0)},\frac{B+E}{A}\right\}<R,
\end{equation}
then the disturbed trajectory remains in the tube and the estimate holds
on all of $[0,t_f]$.
\end{proposition}
\begin{proof}
Relative to the proof of Theorem~\ref{proposition::lyapunov}, the three
additional contributions are bounded by
\[
\frac{\delta_1}{\epsilon_1}b,\qquad
\frac{\delta_2}{\epsilon_2}a,\qquad
\frac{e_2\delta_2}{\epsilon_2}b,
\]
where the last term accounts for motion of the implicit target
$z^{\rm imp}(t,\ell)$. Cauchy--Schwarz bounds their sum by
$E\sqrt{a^2+b^2}$, proving the differential inequality. Scalar
comparison gives the displayed estimate.  Under
\eqref{nonconvex-robust-invariance-condition}, the comparison bound stays
strictly below $R$, and the same first-exit argument as in
Theorem~\ref{proposition::lyapunov} completes the continuation claim.
\end{proof}
\begin{remark}[Interpretation]
The disturbance contribution $E$ separates error entering the
inner primal--dual flow from error entering the outer update. Both enlarge
the local ultimate radius from $B/A$ to $(B+E)/A$.
\end{remark}

\section{Numerical evaluation}\label{sec:applications}
We evaluate the method in four layers that follow the mathematical
development. First, a controlled convex problem isolates the effect of the
directional regularization. Second, we turn the flat saddle-flow
counterexample of Figure~\ref{fig:pd} into a moving nonconvex branch-tracking
problem. Third, a coupled nonconvex test isolates the stability--bias
separation created by the moving proximal center. Fourth, SUMO experiments
examine simulator-in-the-loop deployment at one intersection under
vehicle-throughput and pedestrian-aware objectives.

The convex and traffic trajectories are computed by forward Euler. The
controlled nonconvex tests are instead integrated with high-accuracy
adaptive solvers; in particular, persistence of the conservative flat
oscillator cannot be attributed to Euler amplification or damping.
Continuous-time estimates are not asserted as guarantees for the Euler
implementations. Throughout, the terminal error is the value at the final
recorded iterate; traffic throughput is the number of vehicles discharged
over the stated four-cycle horizon. Dimensions, integration settings,
sampling budgets, and selection rules are reported with the corresponding
experiment.

\subsection{Convex tests}\label{sec:convex-experiment}
We first consider
\begin{equation}
  \min_x\ \norm{x-c(t)}^2
  \qquad\text{subject to}\qquad Ax=0,\quad g(x)\leq0.
  \label{eq:convex-test}
\end{equation}
We compare the directional augmented regularization with Tikhonov
regularization. The independently computed optimizer of the unregularized
problem is used only to measure error; it is not supplied to either flow.
Blue solid curves denote the proposed directional construction, and orange
dashed curves denote Tikhonov regularization.

\begin{figure}[t]
\centering
\begin{minipage}[t]{0.485\textwidth}
  \centering
  \includegraphics[width=\linewidth]{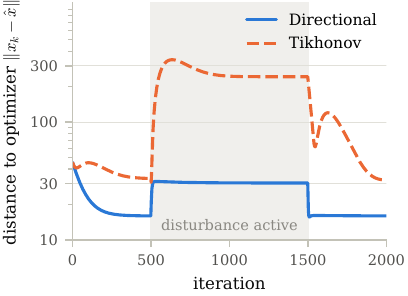}\\[-1mm]
  {\footnotesize (a) Additive disturbance on iterations $500$--$1500$.}
\end{minipage}\hfill
\begin{minipage}[t]{0.485\textwidth}
  \centering
  \includegraphics[width=\linewidth]{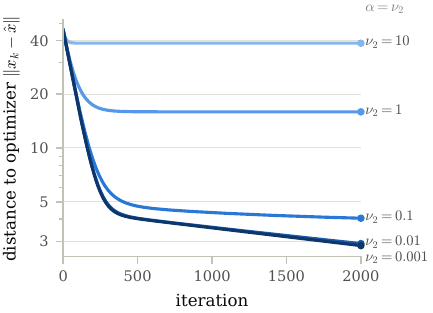}\\[-1mm]
  {\footnotesize (b) Terminal error as $\alpha=\nu_2$ decreases.}
\end{minipage}

\vspace{2mm}
\begin{minipage}[t]{0.485\textwidth}
  \centering
  \includegraphics[width=\linewidth]{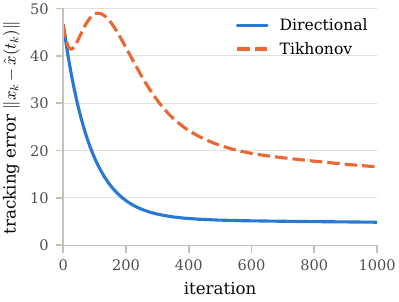}\\[-1mm]
  {\footnotesize (c) Moving target, $\nu_2=\alpha=0.25$.}
\end{minipage}\hfill
\begin{minipage}[t]{0.485\textwidth}
  \centering
  \includegraphics[width=\linewidth]{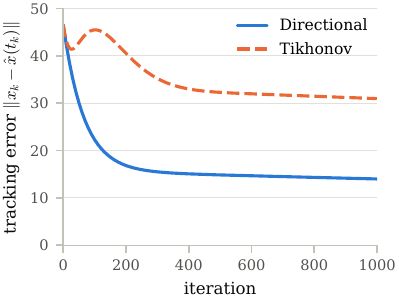}\\[-1mm]
  {\footnotesize (d) Moving target, $\nu_2=\alpha=1$.}
\end{minipage}
\caption{Directional versus Tikhonov regularization on
\eqref{eq:convex-test}. The row-space gain is $\nu_1=20$ in
(a)--(b) and $\nu_1=35$ in (c)--(d). The disturbance in (a) lies in the row
space of $A$, where the equality-residual penalty supplies additional
convexity without moving a feasible optimizer.}
\label{fig:convex-tests}
\end{figure}

In Figure~\ref{fig:convex-tests}(a), we design a disturbance-rejection
experiment by freezing the problem and injecting a sustained row-space
perturbation during iterations $500$--$1499$.  The equality-residual penalty
adds convexity in this direction without moving the feasible optimizer.
During the disturbance, the late-window error is $30.36$ for the directional
flow versus $243.15$ for Tikhonov; after its removal, the terminal errors are
$15.94$ versus $32.16$.  Panel~(b) isolates the associated bias trade-off:
decreasing $\alpha=\nu_2$ from $10$ to $10^{-3}$ lowers the terminal error
from $38.62$ to $2.84$.  Panels~(c)--(d) provide a second test by allowing the
problem data, and hence the optimizer, to move.  The terminal tracking errors
are $4.88$ versus $16.58$ at gain $0.25$
and $14.01$ versus $30.98$ at gain $1$.  Thus the directional construction
rejects exogenous perturbations and tracks endogenous optimizer motion more
effectively, while making its regularization bias explicit.

\subsection{From the flat saddle-flow oscillator to nonconvex tracking}
\label{sec:flat-nonconvex-experiment}
We consider the moving nonconvex problem
\begin{equation}
\begin{aligned}
\min_{(x_1,x_2)\in\real^2}\quad
&f(x_1,x_2)=x_1+\frac{1}{4}(x_2^2-1)^2,\\
\text{subject to}\quad
&x_1-s(t)=0,\qquad s(t)=\frac{3}{4}+\frac{t}{50}.
\end{aligned}
\label{eq:flat-nonconvex-tracking}
\end{equation}
Its two KKT branches are
\[
(x^\star_{1,\pm}(t),x^\star_{2,\pm}(t),p^\star_\pm(t))
=(s(t),\pm1,-1).
\]
The equality fixes the curvature-free coordinate \(x_1\), whereas the
double well selects one of two isolated solutions in the feasible direction.
We track the positive branch.

\paragraph{Why tracking errors persist in the plain primal--dual flow}
The unregularized dynamics are
\begin{equation}
\dot x_1=-(1+p),\qquad
\dot x_2=-x_2(x_2^2-1),\qquad
\dot p=x_1-s(t).
\label{eq:flat-plain-flow}
\end{equation}
To distinguish the frozen KKT point from a trajectory that follows the
moving optimizer, first define the ordinary primal and multiplier errors
relative to the positive KKT branch:
\[
e_x(t):=x_1(t)-s(t),\qquad
e_p(t):=p(t)+1.
\]
Since \(\dot s=1/50\),
\begin{equation}
\begin{aligned}
\dot e_x
 &=\dot x_1-\dot s
   =-(1+p)-\frac{1}{50}
   =-e_p-\frac{1}{50},\\
\dot e_p
 &=\dot p
   =x_1-s(t)
   =e_x.
\end{aligned}
\label{eq:flat-frozen-kkt-errors}
\end{equation}
The constant term in \(\dot e_x\) shifts the center of this oscillator from
\((0,0)\) to \((0,-1/50)\).  Accordingly, define the squared combined error
\[
\mathcal E(t):=e_x(t)^2+\left(e_p(t)+\frac{1}{50}\right)^2.
\]
Using~\eqref{eq:flat-frozen-kkt-errors}, its derivative is
\[
\begin{aligned}
\dot{\mathcal E}(t)
 &=2e_x\dot e_x
   +2\left(e_p+\frac{1}{50}\right)\dot e_p\\
 &=2e_x\left(-e_p-\frac{1}{50}\right)
   +2\left(e_p+\frac{1}{50}\right)e_x
 =0.
\end{aligned}
\]
Consequently,
\begin{equation}
\mathcal E(t)=\mathcal E(0).
\label{eq:flat-plain-invariant}
\end{equation}
If \(\mathcal E(0)>0\), the combined error therefore remains nonzero for all
time.

\paragraph{Regularized two-time-scale flow}
With \(z=(x_1,x_2,p)\) and \(\ell=(\ell_1,\ell_2)\), the proposed dynamics
specialize to
\begin{equation}
\begin{aligned}
\frac{1}{50}\dot x_1
 &=-1-p-2(x_1-\ell_1),&
\frac{1}{50}\dot x_2
 &=-x_2(x_2^2-1)-2(x_2-\ell_2),\\
\frac{1}{50}\dot p
 &=x_1-s(t)-\frac{1}{100}p,&
\frac{1}{10}\dot\ell_i
 &=2(x_i-\ell_i)-\frac{1}{20}\ell_i,
\qquad i\in\{1,2\}.
\end{aligned}
\label{eq:flat-regularized-flow}
\end{equation}
We initialize
\[
(x_1(0),x_2(0),p(0))
=\left(1,\frac{11}{20},-1\right),\qquad
\ell(0)=\left(\frac{40}{41},\frac{22}{41}\right).
\]
Both systems are integrated on \(0\leq t\leq40\) with DOP853, using relative
tolerances \(2\times10^{-11}\) for the plain flow and \(2\times10^{-10}\)
for the regularized flow. The reference curves shown below are evaluated
independently and are used only to measure error.

\begin{figure}[t]
\centering
\includegraphics[width=\textwidth]
  {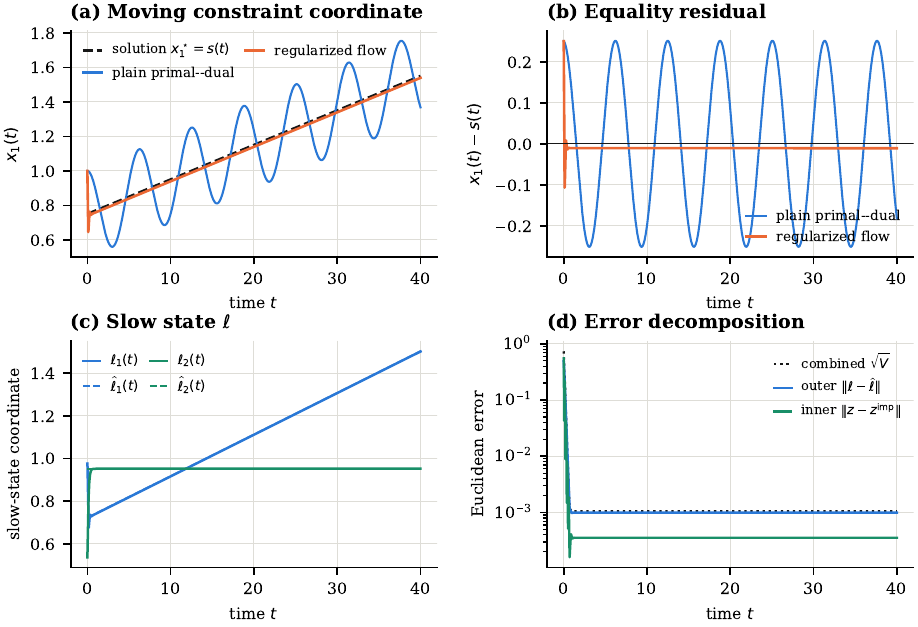}
\caption{Tracking diagnostics for~\eqref{eq:flat-nonconvex-tracking}.
(a) The moving coordinate \(x_1^\star=s(t)\): the plain flow oscillates,
whereas the regularized flow follows its nearby biased target.
(b) The same contrast in the signed equality residual.
(c) The slow state \(\ell\) and its analytic target \(\widehat\ell\).
(d) The outer error \(\|\ell-\widehat\ell\|\), inner error
\(\|z-z^{\rm imp}(t,\ell)\|\), and their combination \(\sqrt V\).}
\label{fig:flat-nonconvex-diagnostics}
\end{figure}

Figure~\ref{fig:flat-nonconvex-diagnostics} shows the mechanism in time.
In panels~(a)--(b), the plain
flow oscillates around the moving solution with constant amplitude, exactly
as predicted by the invariant~\eqref{eq:flat-plain-invariant}; consequently,
neither its primal tracking error nor its equality residual decays. The
regularized flow instead suppresses this oscillation after a short transient
and follows the moving coordinate with a small steady offset. Panels
\ref{fig:flat-nonconvex-diagnostics}(c)--(d) expose the two-time-scale
mechanism behind that behavior: the fast primal--dual state approaches the
instantaneous proximal saddle, while the slow center approaches the selected
moving branch, so both the inner and outer errors become small and remain
bounded as the optimizer moves.

\FloatBarrier
\begin{figure}[t]
\centering
\includegraphics[width=\textwidth]
  {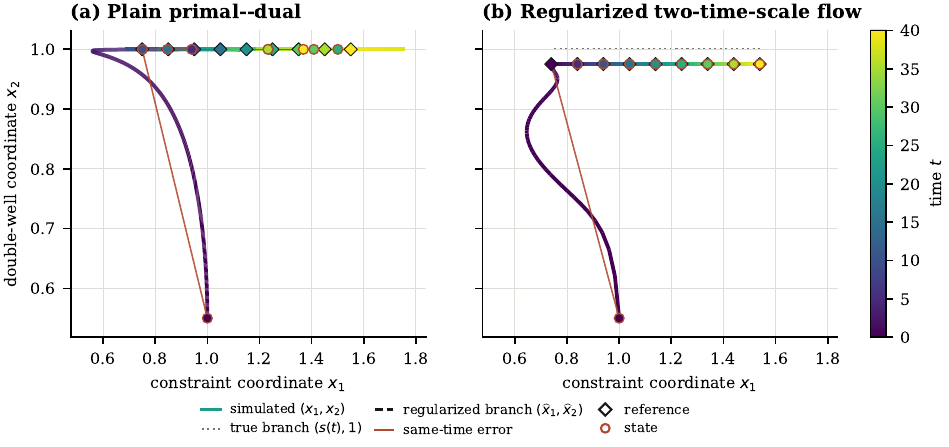}
\caption{The same trajectories in the \((x_1,x_2)\) plane. Diamonds show the
corresponding reference points and circles the simulated states at
\(t=0,5,\ldots,40\); red segments are same-time tracking errors. The plain
trajectory repeatedly advances and retreats along the true \(x_2=1\) branch,
whereas the regularized trajectory settles onto the nearby
\(\widehat x_2=0.9753\) branch and follows it.}
\label{fig:flat-nonconvex-state}
\end{figure}

Figure~\ref{fig:flat-nonconvex-state} gives the geometric counterpart. The
plain trajectory repeatedly advances and
retreats along the positive-well branch, leaving persistent same-time error
segments, whereas the regularized trajectory enters a nearby biased branch
and then advances with it. The remaining displacement from the true branch
is therefore the expected regularization bias, rather than an undamped
tracking oscillation.

\subsection{Coupled nonconvex tracking: stability versus bias}
\label{sec:nonconvex-experiment}
The preceding example is intentionally separable: an ordinary augmentation
can damp its oscillatory coordinate while gradient descent handles its
nonconvex coordinate. We now couple negative curvature in the
constraint-normal direction to a moving double-well branch in order to
isolate why the moving proximal center is useful. Consider
\begin{equation}
\begin{aligned}
\min_{(x_1,x_2)\in\real^2}\quad
&f(x_1,x_2)
=-\frac{3}{2}x_1^2+\frac{1}{4}x_1x_2
  +\frac{1}{2}(x_2^2-1)^2,\\
\text{subject to}\quad
&h_t(x_1,x_2)=x_1-s(t)=0,
\qquad s(t)=-1+\frac{t}{20},\qquad 0\leq t\leq40.
\end{aligned}
\label{eq:coupled-nonconvex-tracking}
\end{equation}
The positive local-solution branch is
\[
x_1^\star(t)=s(t),\qquad
2x_2^\star(t)\bigl((x_2^\star(t))^2-1\bigr)
  +\frac{1}{4}s(t)=0,\qquad
p^\star(t)=3s(t)-\frac{1}{4}x_2^\star(t),
\]
where $x_2^\star(t)$ is the largest real root.

\paragraph{Compared dynamics}
We compare with direct Tikhonov regularization,
\[
\mathcal L_{\nu,\alpha}^{t}(x,p)
=f(x)+p(x_1-s(t))+\frac{\nu}{2}\|x\|^2
-\frac{\alpha^2}{2}p^2.
\]
With the fixed speed scale $1/20$, its saddle flow is
\begin{equation}
\begin{aligned}
\frac{1}{20}\dot x_1&=3x_1-\frac{1}{4}x_2-p-\nu x_1,\\
\frac{1}{20}\dot x_2&=-\frac{1}{4}x_1-2x_2(x_2^2-1)-\nu x_2,\\
\frac{1}{20}\dot p&=x_1-s(t)-\alpha^2p.
\end{aligned}
\label{eq:direct-tikhonov-nonconvex}
\end{equation}
For the proposed regularization, we set
$c=1/5$, $\alpha=1/10$, $\nu=1/20$,
$\epsilon_1=1/1000$, and $\epsilon_2=1/10$ in~\eqref{dynamics}, obtaining
\begin{equation}
\begin{aligned}
\frac{1}{1000}\dot x_1
 &=3x_1-\frac{1}{4}x_2-p-5(x_1-\ell_1),\\
\frac{1}{1000}\dot x_2
 &=-\frac{1}{4}x_1-2x_2(x_2^2-1)-5(x_2-\ell_2),\\
\frac{1}{1000}\dot p
 &=x_1-s(t)-\frac{1}{100}p,\\
\frac{1}{10}\dot\ell_i
 &=5(x_i-\ell_i)-\frac{1}{20}\ell_i,\qquad i\in\{1,2\}.
\end{aligned}
\label{eq:coupled-regularized-flow}
\end{equation}

\begin{figure}[!t]
\centering
\includegraphics[width=\textwidth]
  {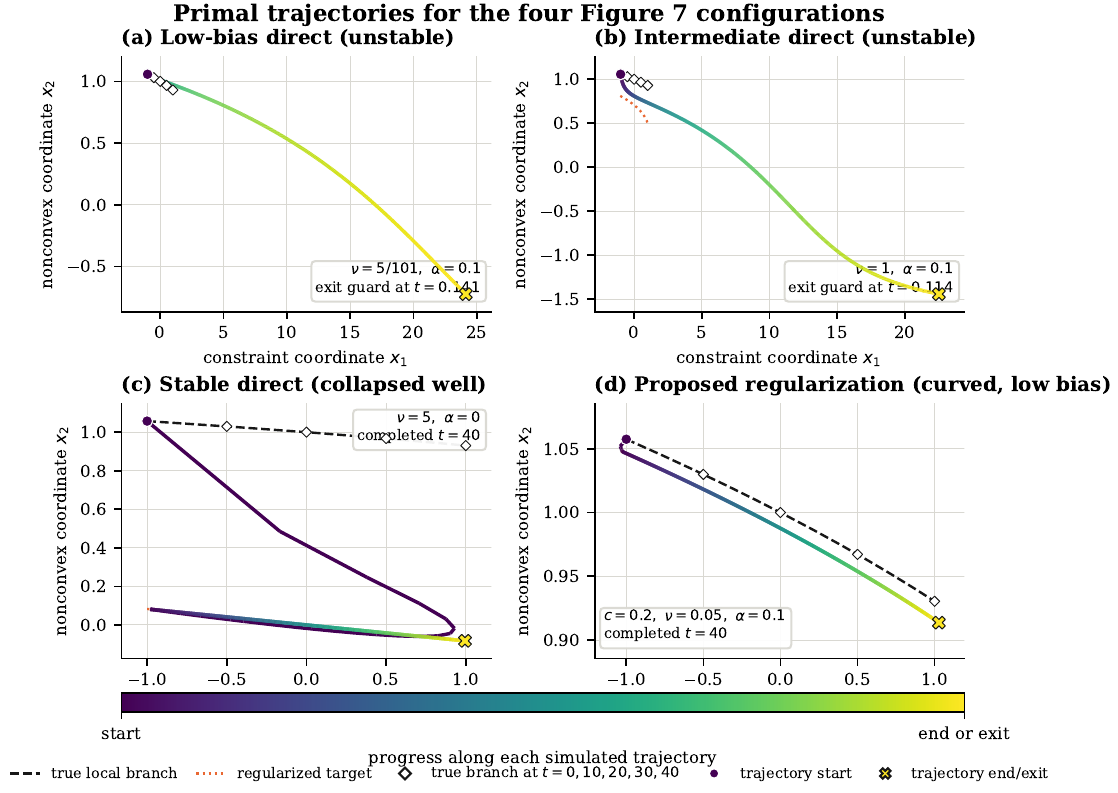}
\caption{Primal trajectories in the $(x_1,x_2)$ state space for the four
configurations.  The corresponding frozen-time landscapes at $t_0=20$ are
shown in Figure~\ref{fig:tikhonov-tradeoff-mechanism}.  Color indicates
progress along each simulated trajectory, the circle marks its initial state,
and the cross marks its terminal state or numerical exit.  Diamonds locate
the true branch at $t=0,10,20,30,40$.  To prevent overflow, the two unstable
direct trajectories in panels (a)--(b) stop when the norm of the full
primal--dual state reaches the exit guard $25$; panels (c)--(d) are integrated
over the entire interval $0\leq t\leq40$.  The stable direct flow suppresses
the instability only by collapsing the desired well, whereas the
proposed dynamics follow the curved positive branch with low bias.}
\label{fig:tikhonov-state-space-trajectories}
\end{figure}

\begin{figure}[!t]
\centering
\includegraphics[width=\textwidth]
  {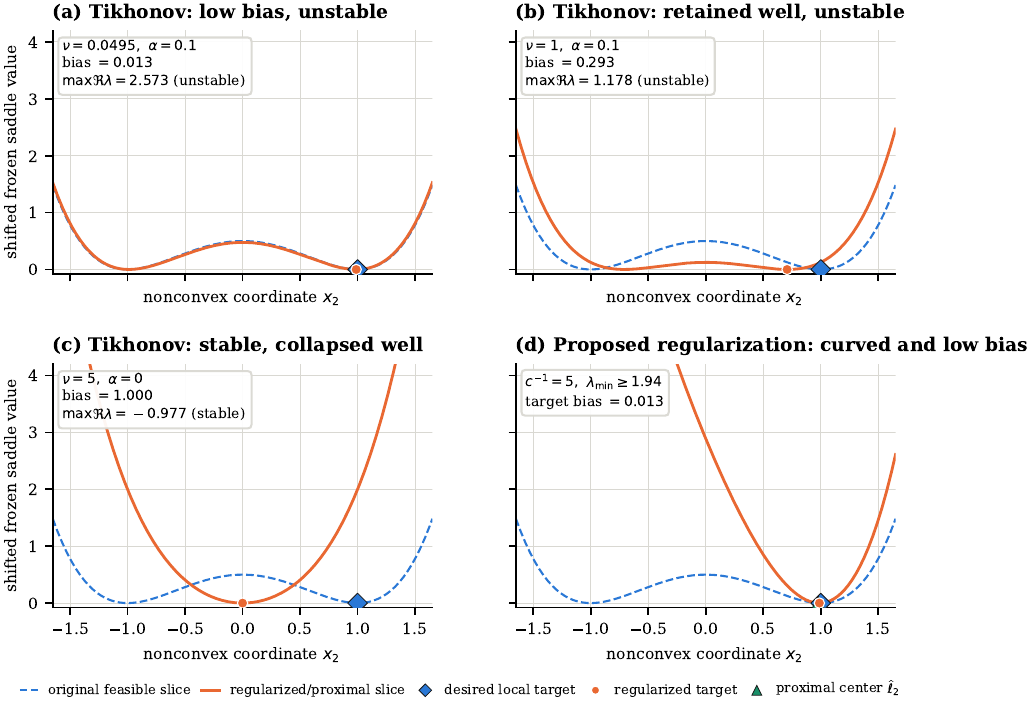}
\caption{Frozen stability--bias mechanism at $s=0$.  Dashed curves are the
original feasible slice and solid curves the direct-Tikhonov or proximal
slice after eliminating $x_1$ and $p$; every curve is shifted by its own
minimum.  (a) The direct parameters sharing the target of the proposed
regularization preserve
the desired well but are unstable.  (b) Increasing $\nu$ reduces instability
while already producing visible target displacement.  (c) The stable direct
tuning fills in the double well and moves the target to $x_2=0$.
(d) A center near the selected branch allows the proximal slice to be strongly
convex while retaining the low-bias target.}
\label{fig:tikhonov-tradeoff-mechanism}
\end{figure}

Figures~\ref{fig:tikhonov-state-space-trajectories}
and~\ref{fig:tikhonov-tradeoff-mechanism} show the resulting four behaviors.
Panels~(a)--(c) expose the stability--bias tradeoff faced by direct Tikhonov
regularization in this example.  The low-bias tuning in panel~(a) nearly
preserves the desired positive branch, with bias $0.01255$, but the trajectory
immediately turns away and exits the local region.  Increasing the direct
regularization in panel~(b) raises the bias to $0.2927$ and changes the escape
path, yet the flow remains unstable.  Only the much stronger tuning in
panel~(c) stabilizes the direct dynamics; by then the positive well has
disappeared, the target has moved to $x_2=0$, and the bias is $1$.  Thus none
of the direct configurations in panels~(a)--(c) is both stable and low bias:
they are either unstable or severely biased.  In contrast, the trajectory
generated by the proposed regularization in panel~(d) remains bounded,
retains bias $0.01255$, and follows
the curved positive branch with a small offset.  The state-space plots display
the dynamical consequence of this tradeoff, while the frozen landscapes show
how the proposed construction retains the desired well while stabilizing
the local motion.

\paragraph{What the frozen picture shows}
We freeze the problem at $t_0=20$, where $s(t_0)=0$ and the selected solution
is $(x_1^\star,x_2^\star,p^\star)=(0,1,-1/4)$.  Figure~\ref{fig:tikhonov-tradeoff-mechanism}
turns each frozen saddle problem into a one-dimensional landscape in $x_2$
by first maximizing over $p$ and minimizing over $x_1$.  The dashed curve is
the original double well and the solid curve is the landscape produced by
the indicated regularization; shifting each curve vertically to have minimum
zero does not change its wells or convexity.  Each panel reports two summary
numbers.  The bias is the distance from the regularized primal target to
$x^\star(t_0)$, while the stability indicator is the largest real part of the
linearized-flow eigenvalues.  A positive indicator means that nearby
perturbations grow; a negative one means that they decay exponentially.  The
calculation isolating the two regularizing terms behind this comparison is
given in Appendix~\ref{app:figure8-landscapes}.

The direct method faces a structural conflict: the same parameter
$\nu$ that supplies missing convexity in the $x_1$ direction also penalizes
the double-well coordinate $x_2$.  This is especially transparent when
$\alpha=0$.  The equality then enforces $x_1=0$, and the direct landscape is
\begin{equation}
\psi_{\nu,0}(x_2)
=\frac{1}{2}(x_2^2-1)^2+\frac{\nu}{2}x_2^2.
\label{eq:tikhonov-frozen-slice}
\end{equation}
For $0\leq\nu<2$, its two wells lie at
$x_2=\pm\sqrt{1-\nu/2}$ and move toward the origin as $\nu$ increases.  They
merge at $\nu=2$.  Thus increasing $\nu$ improves stability in one direction
while simultaneously displacing, and eventually eliminating, the desired
well in the other.

\FloatBarrier

\subsection{Simulator-in-the-loop traffic control}\label{sec:traffic}
\par\noindent
The traffic study asks whether the proposed dynamics can be embedded in a
black-box workflow. SUMO~\cite{PAL-MB-LBW-JE-YE-YF-RH-LL-JR-PW-EW:18}
provides function values, while a differentiable surrogate provides the
local derivatives needed by the dynamics.

\subsubsection{Composite model and implemented dynamics}
Consider the composite optimization problem
\begin{align}
  \begin{split}\label{problem}
    \min_x\quad &f_t(x,\mathcal F(x)), \\
    \text{subject to}\quad &h_t(x,\mathcal F(x))=0,\\
    &g_t(x,\mathcal F(x))\leq0,
  \end{split}
\end{align}
where $\mathcal F$ is known through measurements, data, or simulation rather
than in closed form. Define
\[
\widetilde f_t(x)=f_t(x,\mathcal F(x)),\qquad
\widetilde h_t(x)=h_t(x,\mathcal F(x)),\qquad
\widetilde g_t(x)=g_t(x,\mathcal F(x)).
\]
With the Jacobian convention of Section~\ref{sec:notation}, the chain rule
gives
\begin{align*}
\nabla\widetilde f_t
 &=\nabla_xf_t+\nabla\mathcal F\,\nabla_yf_t,\\
\nabla\widetilde h_t
 &=\nabla_xh_t+\nabla\mathcal F\,\nabla_yh_t,
&
\nabla\widetilde g_t
 &=\nabla_xg_t+\nabla\mathcal F\,\nabla_yg_t,
\end{align*}
where all terms on the right are evaluated at $(x,\mathcal F(x))$.

For reference, the unregularized composite primal--dual dynamics retained
from the original formulation is
\begin{align}
  \begin{split}\label{primal_dual_unknown}
    \dot{x} &=-\nabla\widetilde f_t(x)
    -\nabla\widetilde h_t(x)p-\nabla\widetilde g_t(x)q,\\
    \dot{p} &=\widetilde h_t(x),\\
    \dot{q} &=[\widetilde g_t(x)]^+_q.
  \end{split}
\end{align}
Because the black-box composition need not be convex, the implemented
two-time-scale convexified dynamics is
\begin{align}
  \begin{split}\label{primal_dual_known}
  \epsilon_1\dot{x} &=-\nabla\widetilde f_t(x)
  -\nabla\widetilde h_t(x)p-\nabla\widetilde g_t(x)q
  -\frac{x-\ell}{c},\\
  \epsilon_1\dot{p} &=\widetilde h_t(x)-\alpha^2p,\\
  \epsilon_1\dot{q} &=[\widetilde g_t(x)-\alpha^2q]^+_q,\\
  \epsilon_2\dot{\ell} &=\frac{x-\ell}{c}-\nu\ell.
  \end{split}
\end{align}
Evaluating either vector field requires values of $\mathcal F$ and an
estimate of its Jacobian. In the experiments, SUMO supplies the former and
an RBF surrogate supplies the latter. Figure~\ref{fig:traffic-workflow} separates sampling, surrogate fitting,
optimization, feasibility enforcement, and deployment.
\begin{figure}[tb]
\centering
\resizebox{0.95\textwidth}{!}{%
\begin{tikzpicture}[
  x=1cm,y=1cm,
  box/.style={draw=black!55, rounded corners=2pt, align=center,
    minimum height=0.92cm, text width=2.35cm, inner sep=3pt,
    font=\footnotesize},
  data/.style={box, fill=black!3},
  sim/.style={box, draw=guidegreen!75!black, fill=guidegreen!6},
  model/.style={box, draw=guideorange!85!black, fill=guideorange!6},
  opt/.style={box, draw=guideblue, fill=guideblue!6, very thick},
  deploy/.style={box, draw=guidegreen!75!black, fill=guidegreen!8,
    very thick},
  arrow/.style={-{Latex[length=2mm]}, thick, draw=black!65},
  warm/.style={-{Latex[length=2mm]}, thick, dashed, draw=guideblue}
]
\node[data] (state) at (0,0) {current state\\and cycle demand};
\node[data] (sample) at (3.05,0) {sampled feasible\\timing plans};
\node[sim] (sumosample) at (6.1,0) {SUMO throughput samples\\$\mathcal F_t(s^{(j)})$};
\node[model] (rbf) at (9.15,0) {differentiable\\RBF surrogate};

\node[opt] (flow) at (9.15,-2.0) {two-time-scale\\primal--dual flow};
\node[data] (project) at (6.1,-2.0) {exact feasibility\\projection + diagnostics};
\node[deploy] (commit) at (3.05,-2.0) {commit feasible plan\\to SUMO};
\node[sim] (next) at (0,-2.0) {realized throughput\\and next state};

\draw[arrow] (state) -- (sample);
\draw[arrow] (sample) -- (sumosample);
\draw[arrow] (sumosample) -- (rbf);
\draw[arrow] (rbf) -- (flow);
\draw[arrow] (flow) -- (project);
\draw[arrow] (project) -- (commit);
\draw[arrow] (commit) -- (next);
\draw[arrow] (next.west) -- ++(-0.55,0) |- (state.west)
  node[pos=0.28,left,font=\scriptsize,align=right]{physical state\\feedback};
\draw[warm] (flow.south) -- ++(0,-0.48) -| (sample.south)
  node[pos=0.47,below,font=\scriptsize,text=guideblue]
  {warm start $(x,\ell,p,q)$ at the next cycle};
\end{tikzpicture}
}
\caption{Receding simulator-in-the-loop workflow. Sampling and deployment
start from the same current traffic state. A fresh surrogate is fitted at
each cycle, the dynamical state is warm-started, and the committed timing
vector is projected to the exact cycle-sum and minimum-green constraints
before its realized throughput is recorded.}
\label{fig:traffic-workflow}
\end{figure}
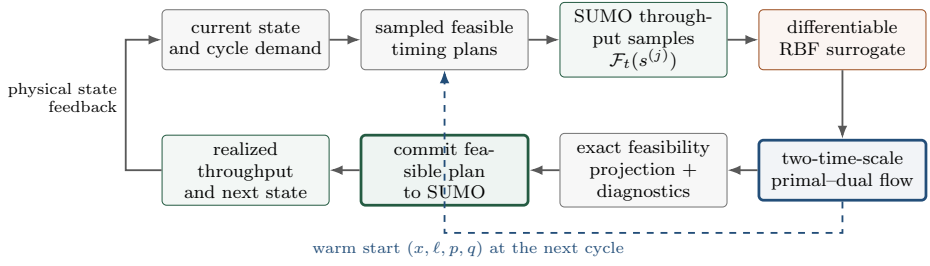

\subsubsection{Single-intersection problem and receding-horizon trajectories}
We consider a controlled intersection with eight admissible signal stages.
The decision is the timing vector $s=(s^1,\ldots,s^8)$. The stages form a
$160$-second cycle and each stage
receives at least $5$ seconds. If $\mathcal F_t(s)$ is the number of vehicles
that leave during cycle $t$, the original timing problem is
\begin{equation}
\begin{array}{lcccr}
  \begin{aligned}
    \min_s &\quad -\mathcal F_t(s),\\
    \text{subject to} &\quad \sum_{i=1}^8s^i=160,\\
    &\quad s^i\geq5,
  \end{aligned}
  &&\text{or}&&
  \begin{aligned}
    \min_{x,s} &\quad -x,\\
    \text{subject to} &\quad x=\mathcal F_t(s),\\
    &\quad \sum_{i=1}^8s^i=160,\\
    &\quad s^i\geq5.
  \end{aligned}
\end{array}
\label{eq:traffic-timing}
\end{equation}
Equivalently, the feasible set is
\[
\mathcal S_1=\{s\in\real^8:\ \mathbf1^Ts=160,\ s_i\geq5\}.
\]
LICQ holds for this equality-plus-bound system. The objective is sampled in
SUMO, fitted with an RBF surrogate, and differentiated only through that
surrogate.

\begin{figure}[!t]
  \centering
  \includegraphics[width=0.47\textwidth,keepaspectratio=true]{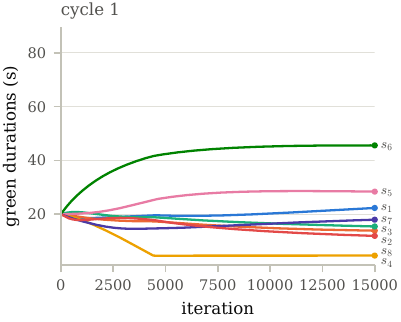}
  \includegraphics[width=0.47\textwidth,keepaspectratio=true]{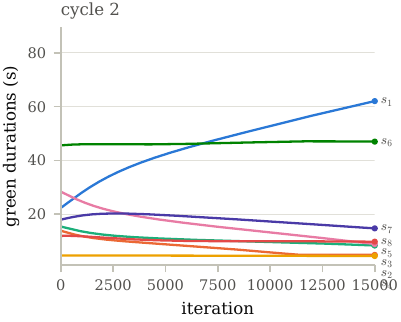}\\
  \includegraphics[width=0.47\textwidth,keepaspectratio=true]{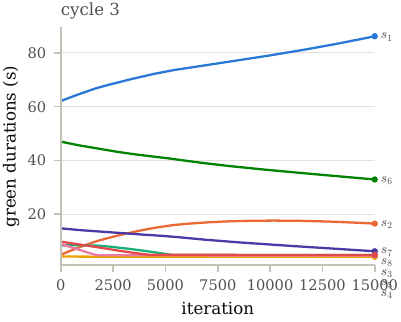}
  \includegraphics[width=0.47\textwidth,keepaspectratio=true]{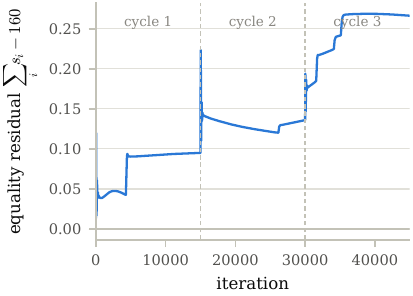}
  \caption{Evolution of the primal variables after three successive
  surrogate updates (top-left: cycle $1$; top-right: cycle $2$;
  bottom-left: cycle $3$). The bottom-right panel reports the equality residual
  $\sum_i s^i-160$ across the three cycles (dashed lines mark the
  surrogate updates); its small nonzero floor is the dual-regularization
  bias $\alpha^2\hat p$.}
  \label{fig:several_steps}
\end{figure}

The online cost of the dynamics is dominated by one surrogate-gradient
evaluation per Euler step, which is closed form for the RBF interpolant.
For the $500$-center RBF used in the diagnostic trajectories of
Figure~\ref{fig:several_steps} ($m=8$), one gradient evaluation costs about
$19\,\mu$s on the reproduction machine and a full primal--dual step costs
about $39\,\mu$s. Thus a complete $15000$-iteration solve takes about
$0.6$\,s, a small fraction of one $160$\,s cycle. This makes per-cycle
re-optimization with warm starts operational; the binding cost is simulator
sampling, a data-acquisition cost shared by any simulation-based method.

More precisely, Figure~\ref{fig:several_steps} uses one seed cloud of
$500$ timing vectors, reevaluated at each of three successive surrogate
updates. At the first update,
$x(0)=\ell(0)=20\,\mathbf 1$, $p(0)=0$, and $q(0)=0$; all four states are then
warm-started at subsequent updates. Each update uses $15000$ simultaneous
forward-Euler steps of size $0.025$, with
\[
c=0.05,\qquad \epsilon_1=1,\qquad \epsilon_2=0.9,
\qquad \alpha=1,\qquad \nu=0.001.
\]
\FloatBarrier

\subsubsection{Controller information and comparison protocol}
\label{sec:traffic-benchmark}
The comparison is deliberately organized around traffic-signal deployment
architectures rather than as a benchmark of general-purpose black-box
optimization algorithms. We consider fixed-time and queue-reactive traffic
controllers. The controller classes do not use the same information. This distinction is
part of the experiment: the proposed policy commits one timing vector for a
cycle, whereas max-pressure reacts to measured queues within the cycle.
Table~\ref{tab:controller-information} states the access model explicitly.

\begin{table}[tb]
\centering
\caption{Controller information and deployment structure. ``Aggregate''
means one throughput value for an entire simulated plan; ``lane state'' means
live per-lane vehicle counts during deployment.}
\label{tab:controller-information}
\small
\begin{tabularx}{\textwidth}{@{}lYYY@{}}
\toprule
Controller & Decision structure & Information used to construct or tune it
& Information required after deployment \\
\midrule
Proposed method
& Fixed eight-stage plan for each cycle
& Aggregate SUMO evaluations of sampled plans
& None within a deployed cycle \\
Equal split / Webster (SUMO)
& Fixed cyclic plan
& Cycle length; Webster also uses route-demand information
& None within a deployed cycle \\
Max-pressure
& Reactive phase selection with a minimum-green setting
& Minimum green selected on a training trace
& Live upstream and downstream lane counts \\
\bottomrule
\end{tabularx}
\end{table}

\subsubsection{Single intersection with full sensing: the benchmark}
\label{sec:plain-benchmark}
Table~\ref{tab:plain-benchmark} reports four traffic-control methods on the
single-intersection instance over four cycles of shifting demand, with the
simulator state carried from one cycle to the next. For the proposed method,
each run samples $80$ timing plans independently in each cycle, and the
reported mean $\pm$ standard deviation is computed over five sampling seeds.
Each cycle uses $15000$ simultaneous forward-Euler steps of size $0.025$, with
\[
c=0.05,\qquad \epsilon_1=1,\qquad \epsilon_2=0.9,
\qquad \alpha=1,\qquad \nu=0.001.
\]
The first cycle starts from $x=\ell=20\,\mathbf 1$, $p=0$, and $q=0$;
all four states are warm-started thereafter. The local sampling radius is
$26$ seconds in the first cycle and $15$ seconds in each subsequent cycle,
and the five sampling seeds are $0$--$4$.
The Webster baseline is generated with SUMO's
\texttt{tlsCycleAdaptation.py} implementation of Webster's formula
\cite{SUMO-TLS:26}.
Max-pressure obtains slightly higher throughput because it continuously
adapts to live traffic conditions within each cycle; among the methods that
commit to a cyclic timing plan, the proposed method performs best.

\begin{table}[tb]
\centering
\caption{Single-intersection benchmark under shifting demand. The proposed
method samples $80$ timing plans independently in each of four cycles; its
entries are the mean $\pm$ standard deviation over five sampling seeds.
Reference controllers are deterministic on this instance.}
\label{tab:plain-benchmark}
\small
\setlength{\tabcolsep}{4pt}
\begin{tabular}{@{}lrl@{}}
\toprule
Controller & Throughput & Sensing \\
\midrule
Max-pressure & $438\pm0$ & live counts, all lanes \\
Proposed & $410\pm12$ & aggregate objective \\
Webster (SUMO) & $317\pm0$ & route demand \\
Equal split & $227\pm0$ & none \\
\bottomrule
\end{tabular}
\end{table}

\subsubsection{A second objective: pedestrians}
\label{sec:pedestrian}
Our methods are able to deal with cost functions that are not amenable to other methods. As an example, we design an scenario where
the intersection now carries pedestrian demand, with its four crossings
separately signalized. We
therefore use the delay-sensitive objective
\[
  F_{\mathrm{veh}}(s)
  -\beta_{\mathrm{ped}}\,\overline W_{\mathrm{ped}}(s),
\]
where $F_{\mathrm{veh}}$ is vehicle throughput and
$\overline W_{\mathrm{ped}}$ is mean accumulated pedestrian waiting time.
The mean includes all pedestrians whose trips have begun by the horizon,
including those whose trips remain unfinished.

For each $\beta_{\mathrm{ped}}\in\{0,0.5,1,2\}$, the dynamics start from
the equal split and reuse the same two cubic-RBF surrogates, fitted from one
shared seed cloud of $600$ SUMO evaluations. Each proposed plan uses
$20000$ simultaneous forward-Euler steps of size $0.01$, with
\[
c=0.05,\qquad \epsilon_1=0.4,\qquad \epsilon_2=0.9,
\qquad \alpha=0.5,\qquad \nu=0.
\]
For every value of $\beta_{\mathrm{ped}}$, the states are initialized at
$x=\ell=20\,\mathbf 1$, $p=0$, and $q=0$.
Table~\ref{tab:pedestrian-waiting}
evaluates the resulting plans on ten unseen SUMO seeds. Across the
tested local sensitivity path, increasing $\beta_{\mathrm{ped}}$ from $0$ to
$2$ reduces mean waiting from $3.13$ to $1.40$\,s, while vehicle throughput
decreases from $33.6$ to $25.1$. The SUMO Webster plan lies inside this
trade-off region: at the level of the reported means, the
$\beta_{\mathrm{ped}}=2$ plan serves more vehicles and has lower pedestrian
waiting. By contrast, vehicle-only max-pressure serves fewer vehicles and
produces much larger pedestrian delay; its phase
selection again freezes at three switches, and all three tested minimum-green
settings give the same training-trace outcome.

\begin{table}[H]
\centering
\caption{Pedestrian-weight sensitivity over one $160$\,s cycle: mean
$\pm$ standard deviation over ten held-out SUMO seeds ($101$--$110$).
The four proposed rows reuse the same $600$ training evaluations;
Webster (SUMO) is constructed from route demand using
\texttt{tlsCycleAdaptation.py}, and max-pressure is re-run live with minimum
green $5$.}
\label{tab:pedestrian-waiting}
\small
\setlength{\tabcolsep}{4pt}
\begin{tabular}{@{}lrr@{}}
\toprule
Controller / weight & Vehicles & Mean wait (s) \\
\midrule
Proposed, $\beta_{\mathrm{ped}}=0$   & $33.6\pm2.2$ & $3.13\pm0.16$ \\
Proposed, $\beta_{\mathrm{ped}}=0.5$ & $32.0\pm3.0$ & $2.97\pm0.17$ \\
Proposed, $\beta_{\mathrm{ped}}=1$   & $31.9\pm2.4$ & $2.49\pm0.19$ \\
Proposed, $\beta_{\mathrm{ped}}=2$   & $25.1\pm2.2$ & $1.40\pm0.11$ \\
\midrule
Webster (SUMO) & $22.3\pm2.3$ & $3.03\pm0.26$ \\
Max-pressure  & $11.1\pm2.0$ & $13.08\pm0.06$ \\
\bottomrule
\end{tabular}
\end{table}

\FloatBarrier

\section{Conclusions}\label{sec:conclusions}
This paper develops a layered continuous-time framework for tracking
solutions of time-dependent constrained problems. A saddle-flow tracking
estimate is made applicable to convex problems through a geometry-aware
regularization, and to nonconvex problems through a local proximal
convexification. Ongoing work addresses:
\begin{itemize}
\item Discrete-time versions of the flows, with tracking and robustness
  bounds under explicit stepsize conditions.
\item Turning the bias--tracking formulas into principled
  parameter-selection rules.
\item Solution-curve bifurcations and active-set changes beyond the
  regular regime considered here.
\item Extending the reproducible benchmark of
  Section~\ref{sec:traffic-benchmark} to larger network instances,
  alternative surrogates, and distributed implementations.
\end{itemize}

\section*{AI-use disclosure}
The authors used generative artificial intelligence tools to assist with
drafting, language editing, and reviewing the manuscript, and with developing
and reviewing code for the numerical simulations. The authors critically
reviewed and verified all AI-assisted material and take full responsibility
for the final content
of the manuscript.

\appendix
\section{Block-matrix derivation for the local envelope}
\label{app:envelope-derivation}
This appendix gives the detailed calculation underlying
Proposition~\ref{lemma::gradient_varphi}; the main text retains only the
proof outline and its role in the construction.

\begin{proof}\leavevmode
Fix $t$ and abbreviate the implicit solution by
$(\inv{x},\inv{p},\inv{q})$.  For a vector $v$, let $v_+$ denote its
componentwise positive part.  The dual optimality conditions for
$\alpha>0$ give
\begin{equation}\label{dual-elimination-envelope}
 \inv{p}=\frac{h_t(\inv{x})}{\alpha^2},
 \qquad
 \inv{q}=\frac{(g_t(\inv{x}))_+}{\alpha^2}.
\end{equation}
Introduce the continuously differentiable, locally $C^{1}$ penalty
\[
 \psi_\alpha^t(x)=f_t(x)
 +\frac{\norm{h_t(x)}^2}{2\alpha^2}
 +\frac{\norm{(g_t(x))_+}^2}{2\alpha^2}.
\]
Evaluating the maximized dual terms by
\eqref{dual-elimination-envelope} shows that
\begin{equation}\label{penalized-envelope-representation}
 \varphi_{c,\alpha}^t(\ell)
 =\psi_\alpha^t(\inv{x})
 +\frac1{2c}\norm{\inv{x}-\ell}^2,
 \qquad
 \nabla\psi_\alpha^t(\inv{x})+\frac{\inv{x}-\ell}{c}=0.
\end{equation}
The second identity in \eqref{penalized-envelope-representation} is not an
additional assumption.  For fixed $(t,\ell,\alpha)$, eliminating the dual
variables reduces the local saddle subproblem to
\[
 \inv{x}\in\operatorname*{argmin}_{x\ \mathrm{near}\ \hat x(t)}
 \left\{\psi_\alpha^t(x)+\frac1{2c}\norm{x-\ell}^2\right\}.
\]
Since the primal variable is unconstrained, its first-order optimality
condition is
\[
 0=\left.\nabla_x\left(
 \psi_\alpha^t(x)+\frac1{2c}\norm{x-\ell}^2\right)\right|_{x=\inv{x}}
 =\nabla\psi_\alpha^t(\inv{x})+\frac{\inv{x}-\ell}{c}.
\]

\begin{enumerate}
\item
\emph{Gradient formula.}
The implicit solution is locally Lipschitz, so the value in
\eqref{penalized-envelope-representation} is locally Lipschitz.  Wherever
$\inv{x}$ is differentiable, differentiating the first identity in
\eqref{penalized-envelope-representation} gives
\begin{equation}\label{eq:regularized-envelope-chain-rule}
\begin{aligned}
 \nabla_\ell\varphi_{c,\alpha}^t(\ell)
 &= (D_\ell\inv{x})^T\nabla\psi_\alpha^t(\inv{x})
    +\frac1c(D_\ell\inv{x}-I)^T(\inv{x}-\ell)\\
 &= (D_\ell\inv{x})^T
    \left(\nabla\psi_\alpha^t(\inv{x})
          +\frac{\inv{x}-\ell}{c}\right)
    +\frac{\ell-\inv{x}}c\\
 &=\frac{\ell-\inv{x}}c.
\end{aligned}
\end{equation}
The final equality follows from the second identity in
\eqref{penalized-envelope-representation}.
The right-hand side is continuous.  At every exceptional point, the
Clarke generalized gradient of the value is therefore the singleton
containing this vector.  Proposition~$2.2.4$ of \cite{FHC:83} then gives
strict differentiability, and the displayed continuous formula proves
that $\varphi_{c,\alpha}^t$ is $\mathcal C^1$.

For $\alpha=0$, this conclusion is precisely
Proposition~\ref{prop::gradient_computation}.  The local implicit
selection in Proposition~\ref{prop:implicit-solution-map} agrees with the
unregularized KKT selection there, so the two gradient formulas are
consistent at $\alpha=0$.

\item
\emph{Uniform local strong convexity.}
We first prove the claim for $\alpha>0$ and then pass to
$\alpha=0$. We account explicitly for changes of active set. Define the reference
and current positive-multiplier sets
\[
 \mathcal A_0(t)=\{i:\hat q^i(t)>0\},
 \qquad
 \mathcal A_+(t,\ell,\alpha)
 =\{i:(\inv{q}(t,\ell,\alpha))^i>0\}.
\]
By continuity, after shrinking the neighborhood,
$\mathcal A_0(t)\subseteq\mathcal A_+(t,\ell,\alpha)$.  Set
\begin{align*}
H&=\nabla^2f_t(\inv{x})
 +\sum_i(\inv{p})^i\nabla^2h_t^i(\inv{x})
 +\sum_i(\inv{q})^i\nabla^2g_t^i(\inv{x}),\\
M_{\mathcal A}&=
 \bigl[\nabla h_t(\inv{x})\ \ 
       \nabla g_t^{\mathcal A}(\inv{x})\bigr],
 \qquad
N_{\mathcal A}=H+\frac{M_{\mathcal A}M_{\mathcal A}^T}{\alpha^2}.
\end{align*}
On every smooth active-set piece, differentiate the equality equations
and  those inequality equations indexed by the \emph{current}
set $\mathcal A=\mathcal A_+(t,\ell,\alpha)$.  The inactive multiplier
derivatives vanish.  The resulting block system is
\[
 \begin{bmatrix}
 H+I/c&M_{\mathcal A}\\
 M_{\mathcal A}^T&-\alpha^2I
 \end{bmatrix}
 \begin{bmatrix}
 D_\ell\inv{x}\\D_\ell(\inv{p},\inv{q}_{\mathcal A})
 \end{bmatrix}
 =
 \begin{bmatrix}I/c\\0\end{bmatrix}.
\]
Its Schur complement gives
\begin{equation}\label{current-active-sensitivity}
 D_\ell\inv{x}=(I+cN_{\mathcal A})^{-1},
 \qquad
 D_\ell\nabla\varphi_{c,\alpha}^t
 =N_{\mathcal A}(I+cN_{\mathcal A})^{-1}.
\end{equation}

The quadratic-form implication of Debreu~\cite[Theorem~$3$]{GD:52}
used below states that, for a symmetric matrix $B$ and any matrix $M$,
\begin{equation}\label{eq:debreu-implication}
\begin{aligned}
 &d^TBd>0\quad\text{for every }0\ne d\in\ker(M^T)\\[-1mm]
 &\qquad\Longrightarrow\quad
 \exists\rho_0>0\ \text{such that}\\[-1mm]
 &\qquad\qquad B+\rho MM^T\succ0
 \quad\text{for every }\rho\geq\rho_0.
\end{aligned}
\end{equation}

Let
$M_0=[\nabla h_t(\inv{x})\ \nabla g_t^{\mathcal A_0(t)}(\inv{x})]$.
At $(\ell,\alpha)=(\hat x(t),0)$,
Assumption~\ref{assumption::seven} gives
$d^THd>0$ for every $0\ne d\in\ker(M_0^T)$.  Applying
\eqref{eq:debreu-implication} therefore makes
$H+\rho M_0M_0^T$ positive definite for every sufficiently large
$\rho$.  Positive definiteness is open, so continuity implies that, on
a smaller parameter neighborhood, there are constants
$\rho_0,\gamma>0$ such that
\begin{equation}\label{uniform-debreu-bound}
 H+\rho_0M_0M_0^T\succeq\gamma I.
\end{equation}
Choose $\alpha$ so that $1/\alpha^2\geq\rho_0$.  Since
$\mathcal A_0(t)\subseteq\mathcal A$, the extra current-active columns
add a positive-semidefinite matrix, and
\[
 N_{\mathcal A}
 \succeq H+\frac{M_0M_0^T}{\alpha^2}
 \succeq\gamma I.
\]
Recalling \eqref{current-active-sensitivity}, every classical Jacobian
of $\nabla\varphi_{c,\alpha}^t$ on a smooth active-set piece satisfies
\[
 D_\ell\nabla\varphi_{c,\alpha}^t
 =N_{\mathcal A}(I+cN_{\mathcal A})^{-1}
 \succeq\frac{\gamma}{1+c\gamma}I=:\sigma I.
\]
To see the last inequality directly, diagonalize the symmetric matrix
$N_{\mathcal A}$.  Each eigenvalue $\lambda\geq\gamma$ is mapped to
$\lambda/(1+c\lambda)$, and the scalar function
$\lambda\mapsto\lambda/(1+c\lambda)$ is increasing on $[0,\infty)$.
There are only finitely many local complementarity faces.  At a switching
point, the Clarke generalized Jacobian is the convex hull of limits of
the corresponding current-active Jacobians.  Equivalently, the generalized
Hessian of the squared-hinge penalty assigns a selector in $[0,1]$ to a
weak constraint, which only adds a positive-semidefinite rank-one term.
The same lower bound is preserved by limits and convex combinations.  The
nonsmooth mean-value theorem
\cite[Proposition~$2.6.5$]{FHC:83} therefore yields
\[
 \left\langle
 \nabla\varphi_{c,\alpha}^t(\ell_1)
 -\nabla\varphi_{c,\alpha}^t(\ell_0),
 \ell_1-\ell_0\right\rangle
 \geq\sigma\norm{\ell_1-\ell_0}^2
\]
whenever the segment lies in the chosen neighborhood.  Hence the envelope
is locally $\sigma$-strongly convex.  This argument allows weakly active
constraints to enter or leave the current positive-multiplier set.

\begingroup
To make the preceding constants uniform, fix a covering time $t_0$ and
use the persistent strongly active set
\[
 \mathcal S_0=\{i:\hat q^i(t_0)>0\}.
\]
The multipliers indexed by $\mathcal S_0$ remain positive on a sufficiently
small time neighborhood of $t_0$. Continuity of the implicit solution then
gives
$\mathcal S_0\subseteq\mathcal A_+(t,\ell,\alpha)$ throughout a common
parameter neighborhood for nearby $t$ and small $\alpha>0$\auditfix{,
using that, by local uniqueness in
Proposition~\ref{prop:implicit-solution-map}, the implicit selection at
$(t,\hat x(t),0)$ coincides with the reference triple, so
$(\inv q)^i$ is near $\hat q^i(t)>0$ for $i\in\mathcal S_0$}.  Apply
\eqref{eq:debreu-implication} to the reference values $H_0$ and
$M_0=M_{\mathcal S_0}$ at $t_0$.  Assumption~\ref{assumption::seven}
is precisely its null-space hypothesis, so there are
$\rho_0,\gamma_0>0$ such that
\[
 H_0+\rho_0M_0M_0^T\succeq2\gamma_0I.
\]
Continuity then preserves
$H+\rho_0M_{\mathcal S_0}M_{\mathcal S_0}^T\succeq\gamma_0I$
on a sufficiently small time and parameter neighborhood. Since every
current active set contains $\mathcal S_0$, its Gram term dominates, so the preceding sensitivity argument supplies a
common positive convexity modulus on the whole patch. This construction
does not require the complete reference active set $\mathcal A_0(t)$ to
remain unchanged when a multiplier crosses zero. A finite subcover of
$[0,t_f]$, followed by taking the smallest modulus and radius, gives
uniform constants for $t\in[0,t_f]$ and
$0<\alpha\leq\bar\alpha$.
\endgroup
By
Propositions~\ref{prop:implicit-solution-map} and
\ref{prop::gradient_computation}, the map
\[
 G(t,\ell,\alpha)=
 \frac{\ell-\inv{x}(t,\ell,\alpha)}c
\]
is continuous also at $\alpha=0$.  Letting $\alpha\downarrow0$ in the
strong-monotonicity inequality shows that $G(t,\cdot,0)$ has the same
local lower bound.

\item
\emph{Critical-point branch.}
Moreover,
$G(t,\hat x(t),0)=0$.  The strong form of the Implicit Function Theorem
for monotone mappings \cite[Theorem~$1$H.$3$]{ALD-RTR:14} now gives a
locally unique, locally Lipschitz root $\hat x_\alpha(t)$ near
$\hat x(t)$.  The equation $G(t,\hat x_\alpha(t),\alpha)=0$ is exactly
\[
 \hat x_\alpha(t)
 =\inv{x}(t,\hat x_\alpha(t),\alpha).
\]
\end{enumerate}
\end{proof}

\section{Derivation of the frozen landscape terms in
Figure~\ref{fig:tikhonov-tradeoff-mechanism}}
\label{app:figure8-landscapes}
At $t_0=20$, one has $s(t_0)=0$, and the objective is
\begin{equation}
 f(x_1,x_2)
 =-\frac32x_1^2+\frac14x_1x_2
   +\frac12(x_2^2-1)^2.
 \label{eq:appendix-frozen-objective}
\end{equation}
The equality constraint is $x_1=0$, so the original feasible slice is
\begin{equation}
 \psi^0(x_2)=f(0,x_2)=\frac12(x_2^2-1)^2.
 \label{eq:appendix-original-slice}
\end{equation}
Subtracting the minimum of any of the following scalar functions only shifts
its graph vertically and does not change its wells or convexity.

\paragraph{Tikhonov term}
For $\alpha>0$, the frozen Tikhonov Lagrangian is
\begin{equation}
 L^{\rm T}_{\nu,\alpha}(x_1,x_2,p)
 =f(x_1,x_2)+px_1+\frac\nu2(x_1^2+x_2^2)
  -\frac{\alpha^2}{2}p^2.
 \label{eq:appendix-tikhonov-lagrangian}
\end{equation}
Maximization over $p$ gives $p^\sharp=x_1/\alpha^2$.  Substitution and
collection of the terms in $x_1$ give
\begin{align}
 \max_p L^{\rm T}_{\nu,\alpha}(x_1,x_2,p)
 &=\frac12(x_2^2-1)^2+\frac\nu2x_2^2
   +\frac{D_{\rm T}}2x_1^2+\frac14x_1x_2,
 \label{eq:appendix-tikhonov-after-p}
\end{align}
where
\begin{equation}
 D_{\rm T}:=\alpha^{-2}+\nu-3.
 \label{eq:appendix-DT}
\end{equation}
Assuming $D_{\rm T}>0$, minimization over $x_1$ gives
\begin{equation}
 x_1^\sharp(x_2)=-\frac{x_2}{4D_{\rm T}}.
 \label{eq:appendix-tikhonov-x1}
\end{equation}
Consequently,
\begin{align}
 \psi^{\rm T}_{\nu,\alpha}(x_2)
 &:=\min_{x_1}\max_p L^{\rm T}_{\nu,\alpha}(x_1,x_2,p) \notag\\
 &=\frac12(x_2^2-1)^2+\frac\nu2x_2^2
   -\frac{x_2^2}{32D_{\rm T}}.
 \label{eq:appendix-tikhonov-landscape}
\end{align}
We now specialize to $\alpha=0$.  As $\alpha\downarrow0$,
$D_{\rm T}\to\infty$, $x_1^\sharp(x_2)\to0$, and
\begin{equation}
 \psi^{\rm T}_{\nu,0}(x_2)
 :=\lim_{\alpha\downarrow0}\psi^{\rm T}_{\nu,\alpha}(x_2)
 =\frac12(x_2^2-1)^2+\frac\nu2x_2^2.
 \label{eq:appendix-tikhonov-undamped-landscape}
\end{equation}
The same expression follows directly at $\alpha=0$: if $x_1\ne0$, then
$\sup_{p\in\mathbb R}px_1=+\infty$, so maximization over $p$ enforces the
equality slice $x_1=0$.

\paragraph{Proposed regularization}
For a fixed center $\ell=(\ell_1,\ell_2)$, the frozen inner Lagrangian is
\begin{align}
 L^{\rm P}_{c,\alpha}(x_1,x_2,p;\ell)
 &=f(x_1,x_2)+px_1
   +\frac{(x_1-\ell_1)^2+(x_2-\ell_2)^2}{2c}
   -\frac{\alpha^2}{2}p^2.
 \label{eq:appendix-proximal-lagrangian}
\end{align}
For $\alpha>0$, maximization over $p$ gives
$p^\sharp=x_1/\alpha^2$.  Substitution and collection of the terms in $x_1$
give
\begin{align}
 \max_p L^{\rm P}_{c,\alpha}(x_1,x_2,p;\ell)
 &=\frac12(x_2^2-1)^2+\frac{(x_2-\ell_2)^2}{2c}
   +\frac{\ell_1^2}{2c} \notag\\
 &\quad+\frac{D_{\rm P}}2x_1^2
   +\left(\frac{x_2}{4}-\frac{\ell_1}{c}\right)x_1,
 \label{eq:appendix-proximal-after-p}
\end{align}
where
\begin{equation}
 D_{\rm P}:=\alpha^{-2}+c^{-1}-3>0
 \label{eq:appendix-DP}
\end{equation}
for the parameters used in Figure~\ref{fig:tikhonov-tradeoff-mechanism}.
Completing the square, or differentiating with respect to $x_1$, gives
\begin{equation}
 x_1^\sharp(x_2;\ell)
 =\frac{\ell_1/c-x_2/4}{D_{\rm P}}.
 \label{eq:appendix-proximal-x1}
\end{equation}
Substitution yields the reduced proximal landscape
\begin{align}
 \psi^{\rm P}_{c,\alpha}(x_2;\ell)
 &:=\min_{x_1}\max_p
 L^{\rm P}_{c,\alpha}(x_1,x_2,p;\ell) \notag\\
 &=\frac12(x_2^2-1)^2
   +\frac{(x_2-\ell_2)^2}{2c}
   -\frac{(x_2/4-\ell_1/c)^2}{2D_{\rm P}}
   +\frac{\ell_1^2}{2c}.
 \label{eq:appendix-proximal-landscape}
\end{align}
The final term is independent of $x_2$.  Thus, writing $\doteq$ for equality
up to an $x_2$-independent additive constant,
\begin{equation}
 \psi^{\rm P}_{c,\alpha}(x_2;\ell)
 \doteq
 \frac12(x_2^2-1)^2
 +\frac{(x_2-\ell_2)^2}{2c}
 -\frac{(x_2/4-\ell_1/c)^2}
 {2(\alpha^{-2}+c^{-1}-3)}.
 \label{eq:appendix-proximal-landscape-shifted}
\end{equation}

The high-level distinction is the center of the quadratic convexity added
directly in the nonconvex coordinate:
\begin{equation}
 \boxed{
 \begin{array}{c|c}
  \text{Tikhonov}
   & \displaystyle \frac{\nu}{2}x_2^2 \\[3mm]
  \text{Moving proximal center}
   & \displaystyle \frac{1}{2c}(x_2-\ell_2)^2
 \end{array}}
 \label{eq:appendix-high-level-comparison}
\end{equation}
Tikhonov convexity pulls toward the origin, whereas the proximal
convexity pulls toward the moving center $\ell_2$, which can remain close to
the selected well.  This is the mechanism emphasized in
Figure~\ref{fig:tikhonov-tradeoff-mechanism}.

\bibliographystyle{siamplain}%
\bibliography{references}

\end{document}